\documentclass[a4paper,UKenglish,cleveref, autoref, nameinlink, thm-restate]{lipics-v2021}

\usepackage{amsmath}
\usepackage{amsfonts}
\usepackage{amssymb}
\usepackage{amsthm}
\usepackage{graphicx}
\usepackage{todonotes}
\usepackage{boxedminipage}
\usepackage{bm}
\usepackage{subcaption}
\usepackage{xspace}
\usepackage{alphalph}

\usepackage{xcolor,soul}
\Crefname{property}{Property}{Properties}
\Crefname{observation}{Observation}{Observations}
\Crefname{theorem}{Theorem}{Theorems}
\Crefname{section}{Section}{Sections}
\Crefname{figure}{Figure}{Figures}
\Crefname{enumi}{Case}{Cases} 

\definecolor{defblue}{rgb}{0.121,0.47,0.705}
\definecolor{orangered}{rgb}{1,0.271,0}
\DeclareTextFontCommand{\emph}{\color{defblue}\em}

\definecolor{lipicsblue}{rgb}{0.08235294118,0.3098039216,0.537254902}

\hypersetup{colorlinks=true,
    linkcolor=orangered,
    anchorcolor=lipicsblue,
    citecolor=lipicsblue,
    filecolor=lipicsblue,
    menucolor=lipicsblue,
    urlcolor=lipicsblue,
    bookmarksopen=true,
    bookmarksopenlevel=2,
    bookmarksnumbered=true,
    plainpages=false,
    }

\newcommand\I{\mathcal I}

\hideLIPIcs

\title{Large Induced Subgraphs of Bounded Degree in Outerplanar and Planar Graphs} 

\author{Marco {D'Elia}}{Roma Tre University, Rome, Italy}{marco.delia@uniroma3.it}{https://orcid.org/0009-0008-6266-3324}{}

\author{Fabrizio Frati}{Roma Tre University, Rome, Italy}{fabrizio.frati@uniroma3.it}{https://orcid.org/0000-0001-5987-8713}{}

\authorrunning{M. {D'Elia} and F. Frati}

\Copyright{Marco D'Elia and Fabrizio Frati}


\category{} 

\relatedversion{} 

\acknowledgements{}

\nolinenumbers 

\begin{document}

\maketitle

\begin{abstract} 
In this paper, we study the following question. Let $\mathcal G$ be a family of planar graphs and let $k\geq 3$ be an integer. What is the largest value $f_k(n)$ such that every $n$-vertex graph in $\mathcal G$ has an induced subgraph with degree at most $k$ and with $f_k(n)$ vertices? Similar questions, in which one seeks a large induced forest, or a large induced linear forest, or a large induced $d$-degenerate graph, rather than a large induced graph of bounded degree, have been studied for decades and have given rise to some of the most fascinating and elusive conjectures in Graph Theory. We tackle our problem when~$\mathcal G$ is the class of the outerplanar graphs or the class of the planar graphs. In both cases, we provide upper and lower bounds on the value of $f_k(n)$. For example, we prove that every $n$-vertex planar graph has an induced subgraph with degree at most $3$ and with $\frac{5n}{13}>0.384n$ vertices, and that there exist $n$-vertex planar graphs whose largest induced subgraph with degree at most $3$ has $\frac{4n}{7}+O(1)<0.572n+O(1)$ vertices. 
\end{abstract}

\section{Introduction}\label{se:introduction}
The study of induced subgraphs has a central role in Graph Theory. As a prominent example, establishing the truth of the Albertson and Berman conjecture~\cite{DBLP:journals/gc/AkiyamaW87,albertson1979conjecture}, stating that every $n$-vertex planar graph contains an induced forest with at least $\frac{n}{2}$ vertices, is one of the most famous and frequently mentioned (see, e.g., \cite{DBLP:journals/jgt/AngeliniEFG16,DBLP:journals/gc/BorradaileLS17,DBLP:journals/dam/DrossMP19,DBLP:journals/dm/DrossMP19,DBLP:journals/combinatorics/LukotkaMZ15,DBLP:journals/dmtcs/KowalikLS10,DBLP:journals/gc/Le18,DBLP:journals/gc/Pelsmajer04,DBLP:journals/gc/Salavatipour06,DBLP:journals/jgt/ShiX17,wxy-ifbpg-17}) open problems in Graph Theory. The conjectured bound $\frac{n}{2}$ is the best possible~\cite{DBLP:journals/gc/AkiyamaW87}. The best known lower bound is $\frac{2n}{5}$, which comes from the acyclic $5$-colorability of planar graphs~\cite{borodin1976proof}. Other popular related conjectures state that every $n$-vertex bipartite planar graph has an induced forest with at least $\frac{5n}{8}$ vertices~\cite{DBLP:journals/gc/AkiyamaW87} and that every $n$-vertex planar graph has an induced linear forest with at least $\frac{4n}{9}$ vertices~\cite{DBLP:journals/gc/Pelsmajer04}. In the former case, the best known lower bound is $\lceil \frac{4n+3}{7}\rceil$~\cite{wxy-ifbpg-17}, while in the latter case it is $\lceil\frac{n}{3}\rceil$~\cite{DBLP:journals/jgt/Poh90}. 

Combinatorial problems on induced graphs usually have the following form: Given a class~$\mathcal G$ of $n$-vertex graphs, what is the largest function $f(n)$ such that every graph in $\mathcal G$ has an induced subgraph with $f(n)$ vertices and satisfying certain properties? Since in a clique one cannot find any induced subgraph other than a clique, the class $\mathcal G$ usually contains sparse graphs -- most often $\mathcal G$ is the class of the $n$-vertex planar graphs or a subclass of it. More specifically, research on this topic has addressed the problem of determining the maximum size of: (i) an induced forest that one is guaranteed to find in a planar graph~\cite{DBLP:journals/gc/AkiyamaW87,albertson1979conjecture,DBLP:journals/jgt/ShiX17}, in an outerplanar graph~\cite{hosono1990induced}, in a $2$-outerplanar graph~\cite{DBLP:journals/gc/BorradaileLS17}, in a planar graph with given girth~\cite{DBLP:journals/jgt/AlonMT01,DBLP:journals/dam/DrossMP19,DBLP:journals/ipl/FertinGR02,DBLP:journals/dmtcs/KowalikLS10,DBLP:journals/gc/Le18,DBLP:journals/gc/Salavatipour06}, in a graph with bounded treewidth~\cite{DBLP:journals/combinatorics/ChappellP13}, or in a cubic graph~\cite{bhs-lb-87,DBLP:journals/jgt/KellyL18,DBLP:journals/dm/Staton84}; (ii) an induced linear forest that one is guaranteed to find in a planar graph~\cite{DBLP:journals/gc/Pelsmajer04,DBLP:journals/jgt/Poh90}, in an outerplanar graph~\cite{DBLP:journals/gc/Pelsmajer04}, or in a planar graph with given girth~\cite{DBLP:journals/dm/DrossMP19}; (iii) a $k$-degenerate induced graph that one is guaranteed to find in a planar graph~\cite{DBLP:journals/combinatorics/DvorakK18,DBLP:journals/jgt/GuKOQZ22,DBLP:journals/combinatorics/LukotkaMZ15}; and (iv) an induced matching that one is guaranteed to find in a planar graph with degree at most $3$~\cite{DBLP:journals/siamdm/KangMM12}.


In this paper, we study the maximum size of an induced graph of bounded degree that one is guaranteed to find in an outerplanar graph or in a planar graph. It was surprising to us, given that the degree is one of the most important, natural, and used graph parameters, that this problem has not been studied systematically before, to the best of our knowledge. A set of vertices which induces a graph of degree at most $k$ in a graph $G$ is sometimes called a \emph{$k$-stable set}~\cite{DBLP:journals/combinatorics/FountoulakisKM10}. Our results are as follows.
\begin{itemize}
	\item We show that every $n$-vertex outerplanar graph has an induced subgraph with degree at most $3$ and with at least $\frac{2}{3}n$ vertices, an induced subgraph with degree at most $k$ and with at least $\frac{2k+1}{2k+5}n$ vertices, for any even integer $k\geq 4$, and an induced subgraph with degree at most $k$ and with at least $\frac{2k+2/3}{2k+14/3}n$ vertices, for any odd integer $k\geq 5$. On the other hand, for any integer $k\geq 3$, there exists an $n$-vertex outerplanar graph, in fact an $n$-vertex outerpath, such that any induced subgraph with degree at most $k$ has at most $\frac{k-1}{k}n + O(1)$ vertices. The last bound is tight for the class of the $n$-vertex outerpaths.
	\item We show that every $n$-vertex planar graph has an induced subgraph with degree at most~$3$ and with at least $\frac{5}{13}n$ vertices, an induced subgraph with degree at most $4$ and with at least $\frac{27}{59}n$ vertices, and, for any integer $k\geq 5$, an induced subgraph with degree at most $k$ and with at least $\frac{k-2}{k+1}n$ vertices. These results actually follow from bounds with respect to the number of vertices {\em and edges} for (not necessarily planar) graphs. On the other hand, for any integer $k\geq 3$, there exists an $n$-vertex planar graph such that any induced subgraph with degree at most $k$ has at most $\frac{2k-2}{2k+1}n + O(1)$ vertices and an $n$-vertex planar graph such that any induced subgraph with degree at most $k$ has at most~$\frac{k+2}{k+4}n + O(1)$~vertices.  
\end{itemize}

\cref{ta:degree-3} shows the coefficients of the linear terms in the size of the induced subgraphs of degree at most $3$ that we can find in the considered graph classes.

\begin{table}[]
	\centering
	\begin{tabular}{|r||c|c|c|c|}
		\hline
		& Forests & Outerplanar graphs & Planar graphs\\ 
		\hline
		LB & $0.75$~\cite{DBLP:journals/combinatorics/ChappellP13} & $0.\overline{6}$~[\cref{th:outerplanar-lb-3}] & $>0.384$~[\cref{th:planar-lb}] \\ \hline
		UB & $0.75$~\cite{DBLP:journals/combinatorics/ChappellP13}  & $0.\overline{6}$~[\cref{th:outerplanar-ub}] & $<0.572$~[\cref{th:planar-ub}] \\ \hline
	\end{tabular}
	\caption{Summary of the best known coefficients of the linear term for $k=3$.}
	\label{ta:degree-3}
\vspace{-7mm}
\end{table}

{\bf Related work.} Although, as far as we know, the question of determining the maximum size of an induced subgraph of bounded degree that one is guaranteed to find in an outerplanar graph or in a planar graph was not addressed explicitly before, some results from the literature have an implication or a relationship to our problem. Interestingly, linear lower bounds on the size of such subgraphs can be derived from known results on different graph parameters. However, our lower bounds have larger multiplicative coefficient of the linear term; making this coefficient as large as possible is the typical objective in this research field, think again about the Albertson and Berman conjecture~\cite{DBLP:journals/gc/AkiyamaW87,albertson1979conjecture} and the related results mentioned above.

\begin{itemize}
\item For $k=0$, a $k$-stable set is just called \emph{stable set} or \emph{independent set}; the literature about finding large independent sets in planar graphs is one among the richest in the entire field of graph theory, see, e.g.~\cite{ALBERTSON197684,appel1976every1,appel1976every2,DBLP:journals/mst/BiedlW05,caro1985vertex,DBLP:journals/jct/RobertsonSST97}. In particular, a consequence of the four-color theorem is that every planar graph admits an independent set with $\frac{n}{4}$ vertices. This defines an induced graph whose degree is bounded by $k$, for any $k\geq 0$. 

\item Fountoulakis, Kang, and McDiarmid~\cite{DBLP:journals/combinatorics/FountoulakisKM10} studied the maximum size of an induced subgraph of bounded degree that one is expected to find in a random graph. 

\item A \emph{$(c,k)$-defective coloring} of a graph is a coloring of the vertices with $c$ colors such that each vertex $v$ has at most $k$ neighbors with the same color as $v$. Since the vertices in each color class induce a graph whose degree is bounded by $k$, the existence of a $(c,k)$-defective coloring for a graph $G$ implies the existence of a set of~$\frac{n}{c}$ vertices that induce a subgraph of~$G$ with degree at most $k$. Cowen, Cowen, and Woodall~\cite{DBLP:journals/jgt/CowenCW86} proved that every outerplanar graph has a $(2,2)$-defective coloring and that every planar graph has a $(3,2)$-defective coloring. These results imply the existence of sets with $\frac{n}{2}$ and $\frac{n}{3}$ of vertices in an $n$-vertex outerplanar or planar graph, respectively, that induce graphs with degree at most $2$. Our lower bounds are larger than these values, already for $k=3$ in which they are $\frac{2}{3}n$ and $\frac{5}{13}n$, respectively. See~\cite{DBLP:journals/combinatorics/wood18} for many more results on $(c,k)$-defective colorings.

\item Linear forests have degree at most $2$, hence the size of an induced linear forest that one is guaranteed to find in an $n$-vertex outerplanar graph (the optimal size $\lceil \frac{4n+2}{7}\rceil$ was established by Pelsmajer~\cite{DBLP:journals/gc/Pelsmajer04}) or in an $n$-vertex planar graph (the best known lower bound $\lceil\frac{n}{3}\rceil$ is due to Poh~\cite{DBLP:journals/jgt/Poh90}) is a lower bound for the size of an induced graph with degree at most $k$, for any integer $k\geq 3$, that one is guaranteed to find in an $n$-vertex outerplanar or planar graph, respectively. Our lower bounds are larger than these values, already for $k=3$ in which they are $\frac{2}{3}n$ and $\frac{5}{13}n$, respectively.
 
\item Chappell and Pelsmajer~\cite{DBLP:journals/combinatorics/ChappellP13} proved that every $n$-vertex graph with treewidth $t$ contains an induced {\em forest} with degree at most $k$ and with at least $\lceil \frac{2kn +2}{tk+k+1} \rceil$ vertices. This provides a lower bound of $\lceil \frac{2kn +2}{3k+1} \rceil$ on the size of an induced subgraph with degree at most $k$ that one is guaranteed to find in an $n$-vertex outerplanar graph. Our lower bounds are larger than $\lceil \frac{2kn +2}{3k+1} \rceil$ for any $k\geq 3$ (and sufficiently large $n$).

\item Chappell and Pelsmajer~\cite{DBLP:journals/combinatorics/ChappellP13} also proved (and attributed the result to Chappell, Gimbel, and Hartman) that every $n$-vertex forest contains an induced forest with degree at most $k$ and with at least $\lceil \frac{k+1}{k+2}n \rceil$ vertices. This bound is the best possible.  

\item It can be proved by edge-counting arguments that sparse graphs, like planar graphs, have many vertices (and hence large induced subgraphs) with bounded degree. The linear lower bounds one gets in this way are worse than those of our paper, see, e.g.,~\cite{DBLP:journals/siamcomp/Kirkpatrick83,DBLP:conf/esa/SnoeyinkK97}. In the same spirit, Bose, Dujmovi\'c, and Wood~\cite{DBLP:journals/cdm/BoseDW06} considered the following problem. Let~$\mathcal G_{n,w}$ be the family of the $n$-vertex graphs with treewidth $w$, where $n\geq 2w+1$. For given integers $0\leq t\leq w$ and $k\geq 2w$, what is the largest size of an induced graph that one is guaranteed to find in every graph $G$ in $\mathcal G_{n,w}$ such that the treewidth of the induced graph is at most $t$ and the degree of its vertices in $G$ is at most $k$? They provided a bound for this problem which implies that every $n$-vertex outerplanar graph $G$ (in fact, every graph with treewidth~$2$) has a subset of at least $\frac{k-3}{k-1}n$ vertices whose degree in $G$ is at most $k$. This lower bound is weaker than our lower bounds for every integer $k\geq 3$. 

\item Knauer and Ueckerdt~\cite{DBLP:journals/corr/abs-2303-13655} considered the following problem. For given integers $w\geq 1$ and $c\geq 1$, what is the largest size of an induced graph that one is guaranteed to find in every graph of treewidth at most $w$ such that each connected component of the induced graph has at most $c$ vertices? They provided a $\frac{c}{w+c+1}n$ lower bound and a $\frac{c}{w+c}n$ upper bound for this problem. Their results imply (by setting $c=k+1$) a $\frac{k+1}{w+k+2}n$ lower bound for the size of an induced graph of degree at most $k$ that one is guaranteed to find in every graph of treewidth at most $w$. In particular, for outerplanar graphs this yields a $\frac{k+1}{k+4}n$ lower bound, which is weaker than our lower bounds for every integer $k\geq 3$.

\item Finally, as mentioned above, the size of $k$-degenerate subgraphs in planar graphs, for $k=2$~\cite{DBLP:journals/combinatorics/DvorakK18}, $k=3$~\cite{DBLP:journals/jgt/GuKOQZ22}, and $k=4$~\cite{DBLP:journals/combinatorics/LukotkaMZ15} has been studied. A lower bound on the size of an induced subgraph with degree at most $k$ is also a lower bound on the size of a $k$-degenerate induced subgraph, and an upper bound for the latter size is an upper bound for the former one. The upper and lower bounds in~\cite{DBLP:journals/jgt/GuKOQZ22,DBLP:journals/combinatorics/LukotkaMZ15} are in fact larger than ours. 
\end{itemize}

The rest of the paper is organized as follows. In \cref{se:preliminaries} we give some preliminaries; in \cref{se:outerplanar} and \cref{se:planar},  we present our results for outerplanar graphs and planar graphs, respectively; finally, in \cref{se:conclusions}, we conclude and present some open problems. 




\section{Preliminaries}\label{se:preliminaries}
We use standard terminology in Graph Theory \cite{Diestelbook}. For a graph $G$, we let $V(G)$ and $E(G)$ be its vertex and edge sets, respectively. For a set $\I \subseteq V(G)$, we denote by $G[\I]$ the \emph{induced subgraph} of $G$, that is, the graph whose vertex set is $\I$ and whose edge set consists of every edge $uv\in E(G)$ such that $u,v\in \I$. The \emph{degree of a vertex} $v$ is the number of its incident edges, and the \emph{degree of a graph} is the maximum degree of any of its vertices. A set $\mathcal{D} \subseteq V(G)$ is a \emph{dominating set} if every vertex of $G$ is in $\mathcal{D}$ or is adjacent to a vertex in $\mathcal D$.

A \emph{drawing} of a graph maps each vertex to a point in the plane and each edge to a Jordan arc between the points representing the end-vertices of the edge. A drawing is \emph{planar} if no two edges cross, except at common end-vertices. A \emph{planar graph} is a graph that admits a planar drawing. A planar drawing partitions the plane into connected regions, called \emph{faces}; the unique unbounded face is the \emph{outer face}, while the other faces are \emph{internal}. 
%
%
A planar drawing is \emph{outerplanar} if every vertex is incident to the outer face. An \emph{outerplanar graph} is a graph that admits an outerplanar drawing. A \emph{maximal outerplanar graph} is an outerplanar graph such that adding any edge to it results in a graph that contains a multiple edge or is not outerplanar. Consider a maximal outerplanar graph $G$ and observe that $G$ has a unique outerplanar drawing, up to a reflection and a homeomorphism of the plane. We call \emph{outerplane embedding}, and denote it by $\mathcal O_G$, the topological information associated with such a drawing. Without loss of information, one can think that $\mathcal O_G$ is just an outerplanar drawing of $G$. Observe that the outer face of $\mathcal O_G$ is delimited by a Hamiltonian cycle, while the internal faces are delimited by cycles with $3$ vertices. An edge $uv$ of $G$ splits $G$ into two outerplanar graphs, namely the subgraphs of $G$ that are induced by the vertices encountered when walking in clockwise and counter-clockwise direction along the outer face of $\mathcal O_G$ from $u$ to $v$. These are called \emph{$uv$-split subgraphs}  of $G$. The vertices $u$ and $v$, as well as the edge $uv$, belong to both $uv$-split subgraphs of $G$. If $uv$ is incident to the outer face of $\mathcal O_G$, then one of these subgraphs is just the edge $uv$ and the other one is $G$. We define the \emph{distance} in $\mathcal O_G$ between two distinct edges $uv$ and $xy$ as the number of steps that are needed to reach $xy$ from $uv$, where a step moves an edge $wz$ (initially $wz:=uv$) to an edge that appears together with $wz$ on the boundary of an internal face  of $\mathcal O_G$. Thus, for example, if $uv$ and $xy$ are incident to the same internal face of $\mathcal O_G$, their distance in $\mathcal O_G$ is $1$. The \emph{$xy$-parent} of the edge $uv$ is the edge $rs$ that appears together with $uv$ on the boundary of an internal face  of $\mathcal O_G$ and that is such that the distance between $xy$ and $uv$ in $\mathcal O_G$ is $1$ plus the distance between $xy$ and $rs$ in $\mathcal O_G$. The \emph{weak dual} of $\mathcal O_G$ is the graph with a vertex for each internal face of $\mathcal O_G$ and with an edge between two vertices if the corresponding faces of $\mathcal O_G$ share an edge. The weak dual of $\mathcal O_G$ is a tree with degree at most $3$. We say that $G$ is a \emph{maximal outerpath} if its weak dual is a path. An \emph{outerpath} is a subgraph of a maximal outerpath.

\section{Induced Subgraphs of Outerplanar Graphs} \label{se:outerplanar}

In this section, we consider the size of the largest induced subgraph of degree at most $k$ that one is guaranteed to find in an $n$-vertex outerplanar graph. \cref{ta:outerplanar} shows the coefficients of the linear terms in the size of the induced subgraphs we can find, for various values of $k$.
\begin{table}[htb]
\centering
\begin{tabular}{|r||c|c|c|c|c|}
\hline
   & $k=3$ & $k=4$ & $k=5$ & $k=6$ & $k=7$ \\ \hline
LB & $0.\overline{6}$     & $>0.69$     & $0.\overline{72}$     & $>0.76$   & $>0.78$     \\ \hline
UB & $0.\overline{6}$     & $0.75$     & $0.8$     & $0.8\overline{3}$  & $<0.858$     \\ \hline
\end{tabular}
\caption{Results on outerplanar graphs with different values of $k$.}
\vspace{-7mm}
\label{ta:outerplanar}
\end{table}

We start with an upper bound, which comes from a family of outerpaths, see~\cref{fig:upper-bound-outerpath}.

\begin{figure}
    \centering
    \includegraphics[scale=.7, page=1]{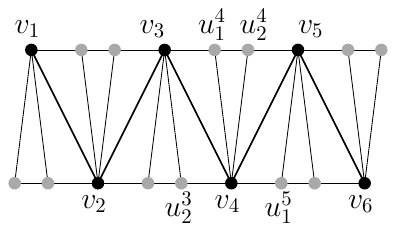}
    \caption{Illustration for~\cref{th:outerplanar-ub}, with $n=18$, $k=3$, and $h=6$.}
    \label{fig:upper-bound-outerpath}
\end{figure}

\begin{theorem} \label{th:outerplanar-ub}
For every $n\geq 1$ and for every integer $k\geq 3$, there exists an $n$-vertex outerpath $G$ with the following property. Consider any set $\mathcal I \subseteq V(G)$ such that $G[\mathcal I]$ has degree at most~$k$. Then $|\mathcal I|\leq \frac{k-1}{k}n+2$.
\end{theorem}

\begin{proof}
Let $G$ be the $n$-vertex outerpath depicted in~\cref{fig:upper-bound-outerpath} and defined as follows. Let $h=\lfloor \frac{n}{k}\rfloor$. The graph contains a path $v_1,\dots,v_h$, vertices $u^i_{1},\dots,u^i_{k-1}$, for $i=1,\dots,h$, and edges $u^i_{j}u^i_{j+1}$, $v_iu^i_j$, $v_{i-1}u^i_{1}$, and $u^i_{k-1}v_{i+1}$, whenever the end-vertices of such edges are in $G$. Additional $n-h\cdot k$ isolated vertices, which do not play any role in the proof, are added to $G$, so that it has $n$ vertices. 

Consider any set $\mathcal I \subseteq V(G)$ such that $G[\mathcal I]$ has degree at most $k$. We prove that there exists a set $\mathcal I' \subseteq V(G)$ such that $G[\mathcal I']$ has degree at most $k$, such that $|\mathcal I'|=|\mathcal I|$, and such that $\mathcal I'$ does not contain any of $v_2,\dots,v_{h-1}$. Suppose that $\mathcal I$ contains a vertex $v_i$, for some $i\in \{2,\dots,h-1\}$. Then there exists a vertex $w$ among $u^{i-1}_{k-1},u^i_{1},u^i_{2},\dots,u^i_{k-1},u^{i+1}_{1}$ that is not in $\mathcal I$, as otherwise $v_i$ would have degree $k+1$ in $G[\mathcal I]$. Replace $v_i$ in $\mathcal I$ with $w$. Clearly, the size of $\mathcal I$ remains the same. Also, $G[\mathcal I]$ still has degree at most $k$. Indeed, the only vertices that have degree larger than $k$ in $G$ are $v_1,\dots,v_h$. Also, $w$ is a neighbor of $v_i$, which however is not in $\mathcal I$ after the replacement, and might be a neighbor of $v_{i-1}$ or $v_{i+1}$. However, $v_{i-1}$ or $v_{i+1}$ are also neighbors of $v_i$, hence if they are in $\mathcal I$, their degree in $G[\mathcal I]$ does not increase. The repetition of this replacement results in the desired set $\mathcal I'$. 

It remains to observe that $|\mathcal I'|\leq n-(h-2)=n-\lfloor \frac{n}{k}\rfloor+2\leq \frac{k-1}{k}n+2$. 
\end{proof}

The upper bound of \cref{th:outerplanar-ub} was proved by considering $n$-vertex outerpaths. For this graph class, the bound is tight, up to additive constants, as proved in the following theorem.

 \begin{theorem} \label{th:outerpath-lb}
For every $n\geq 1$, for every $n$-vertex outerpath $G$, and every integer $k \geq 3$, there exists a set $\mathcal I \subseteq V(G)$ such that $G[\mathcal I]$ has degree at most $k$ and $|\mathcal I|\geq \frac{k-1}{k}n$.
\end{theorem}

\begin{proof}


Without loss of generality, we can assume that $G$ is a maximal outerpath, as otherwise edges can be added to it so that it becomes maximal, and then a subset $\mathcal I$ with the properties required by~\cref{th:outerpath-lb} in the augmented graph satisfies the same properties in the original graph. Let $e$ be an edge that is incident both to the outer face of the outerplane embedding~$\mathcal O_G$ of $G$ and to an internal face corresponding to an extreme of the weak dual of $\mathcal O_G$; we refer to $e$ as to the \emph{special edge}. Also, let $u$ be any of the two end-vertices of $e$, and let $v$ be the other one; we refer to $u$ as to the \emph{special vertex}. We prove, by induction on $n$, that there exists a set $\mathcal{I}\subseteq V(G)$ such that: (i) $|\mathcal{I}| \geq \frac{k-1}{k}n$; (ii) $G[\I]$ has degree at most $k$; and (iii) if the special vertex $u$ is in $\mathcal{I}$, then $u$ has at most $k-1$ neighbors in $\mathcal{I}$. This implies the statement of the theorem. 

There are two base cases:
\begin{itemize}
    \item {\bf Case 1: $n \leq k$.} In this case, we let $\mathcal{I}:=V(G)$. Then $|\mathcal{I}| = n > \frac{k-1}{k}n$. Since there are at most $k$ vertices in $\mathcal{I}$, we have that $G[\I]$ has degree at most $k-1$.

    \item {\bf Case 2: $n = k + 1$.} Let $x$ be any vertex of $G$. We let $\mathcal{I}:=V(G)\setminus \{x\}$. Then $|\mathcal{I}| = k>\frac{(k-1)(k+1)}{k}=\frac{k-1}{k}n$. Since there are at most $k$ vertices in $\mathcal{I}$, we have that $G[\I]$ has degree at most $k-1$.
\end{itemize}


In the inductive case, we have $n \geq k+2$. Our strategy is as follows. We identify an edge $wz$ of $G$ with the following properties. Let $G_R$ be the $wz$-split subgraph of $G$ not containing~$uv$. Also, let $G_L$ be the subgraph of $G$ induced by $V(G)\setminus V(G_R)$. We require that $G_L$ has at least $k$ vertices and that it contains a vertex $h$ such that, by defining $\I=\I_L\cup \I_R$, where $\I_L:=V(G_L)\setminus \{h\}$ and $\I_R\subseteq V(G_R)$ is the set obtained by applying induction on $G_R$, the conditions of the claim are satisfied. We now make this argument precise. 




\begin{figure}[htb]
    \centering
	\begin{subfigure}[b]{0.30\textwidth}
		\centering
		\includegraphics[scale=.8, page=1]{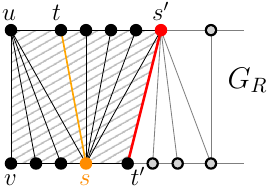}
		\subcaption{}
	\end{subfigure}
	\hfill
	\begin{subfigure}[b]{0.30\textwidth}
		\centering
		\includegraphics[scale=.8, page=2]{img/outerpathPages.pdf}
		\subcaption{}
	\end{subfigure}
    \hfill
	\begin{subfigure}[b]{0.30\textwidth}
		\centering
		\includegraphics[scale=.8, page=3]{img/outerpathPages.pdf}
		\subcaption{}
	\end{subfigure}
	\caption{Illustration for the proof of \cref{th:outerpath-lb}. Graph $G_L$ is striped, vertices and edges in $G_R$ are either not shown or colored gray, the old active vertex $s$ and edge $st$ are orange, while the new active vertex $s'$ and edge $s't'$ are red. (a) All the neighbors of $s$ are moved into $G_L$, which still has at most $k+1$ vertices. (b) All the neighbors of $s$ are moved into $G_L$, which now has at least $k+2$ vertices. (c) The desired vertex $h$ is set to be $s$, the desired edge $wz$ is set to be $s't'$, $s'$ and $t'$ are moved to $G_R$, $\mathcal{I_L} = V(G_L) \setminus \{h\}$, and induction is applied on $G_R$ with $wz$ as a special edge.}
    \label{fig:lower-bound-outerpath}
\end{figure}

In order to find the desired edge $wz$, subgraph $G_L$, and vertex $h$, we work iteratively. We initialize $G_L$ to consist of the vertices $u$ and $v$ (and of the edge $uv$), and $G_R$ to $G[V(G)\setminus \{u,v\}]$. Notice that only one of $u$ and $v$ has at least two neighbors in $G_R$; such a vertex is called \emph{active vertex} and $uv$ is the \emph{active edge}. While $G_L$ has at most $k+1$ vertices, let $s$ be the current active vertex and let $st$ be the current active edge. We move to $G_L$ all the neighbors of $s$ that are in $G_R$. Let $s'$ and $t'$ be the neighbors of $s$ such that the edge $s't'$ exists and its distance from $uv$ in $\mathcal O_G$ is maximum. Now two cases are possible. First, if $G_L$ still has at most $k+1$ vertices, as in \cref{fig:lower-bound-outerpath}(a), we continue the iteration. Only one of $s'$ and $t'$ has at least two neighbors in $G_R$; such a vertex is the new active vertex and $s't'$ is the new active edge.  Conversely, if $G_L$ has at least $k+2$ vertices, as in \cref{fig:lower-bound-outerpath}(b), we stop the iteration, we move $s'$ and $t'$ back to $G_R$, as in \cref{fig:lower-bound-outerpath}(c), we define the desired vertex $h$ to be $s$, the desired edge $wz$ to be $s't'$, and the desired graph $G_L$ to be the current graph $G_L$. We define $\mathcal I_L=V(G_L)\setminus \{h\}$. We apply induction on $G_R$ with $wz$ as a special edge and with the one between $w$ and $z$ that has more than two neighbors in $G_R$ as a special vertex. The induction returns us a set $\I_R\subseteq V(G_R)$ and we define $\I:=\I_L\cup \I_R$. 

We now show that the set $\I$ satisfies the required properties.
\begin{itemize}
\item First, we prove that $|\I|\geq \frac{k-1}{k}n$. By induction, we have $|\I_R|\geq \frac{k-1}{k}|V(G_R)|$. Also, $|V(G_L)|$ is at least $k$, given that it is at least $k+2$ before $s'$ and $t'$ are moved back to $G_R$. Since $|\I_L|=|V(G_L)|-1$, we have that $|\I_L|\geq \frac{k-1}{k}|V(G_L)|$. Hence $|\I|=|\I_L|+|\I_R|\geq \frac{k-1}{k}(|V(G_L)|+|V(G_R)|)=\frac{k-1}{k}n$. 
\item Second, we prove that, if $u\in \I$, then $u$ has at most $k-1$ neighbors in $\mathcal{I}$. Notice that $u$ can have at least $k$ neighbors in $\mathcal{I}$ only if it has degree at least $k\geq 3$. This implies that $u$ is selected as the active vertex in the first step of the iteration for the definition of $G_L$. If $u$ has degree at least $k+1$, then the iteration defining $G_L$ stops at the first step by inserting in $G_L$ all the neighbors of $u$, which implies that $h=u$ and that $u$ is not in $\I$. However, if $u$ has degree $k$, then after inserting in $G_L$ all the neighbors of $u$, we have that $G_L$ has $k+1$ vertices. Now the active vertex becomes one of $u$'s neighbors, which is then selected as $h$, and hence not inserted in $\I$, given that the second step of the iteration results in $G_L$ having at least $k+2$ vertices. Hence, $u$ has $k-1$ neighbors in $\I$. 
\item Finally, we prove that $G[\I]$ has degree at most $k$. Since $w$ and $z$ are the only vertices of $G_R$ with neighbors in $G_L$, we have that any vertex in $V(G_R)\setminus \{w,z\}$ has degree at most $k$ in $G[\I]$, by induction. Next, consider the graph $G_L$ and the active edge $st$ in the step of the iteration defining $G_L$ before the last one, i.e., before all the neighbors of $s$ in $G_R$ are moved to $G_L$ and two of them are moved back to $G_R$ since $G_L$ has at least $k+2$ vertices. At that point, $G_L$ has at most $k+1$ vertices, hence all its vertices other than $s$ and $t$ have degree at most $k$ in $G$ and hence also in $G[\I]$. Furthermore, $s$ is not in~$\I$, since at the next step of the iteration it is selected as the vertex $h$. Also, $t$ has at most one neighbor not in $G_L$, since it is not the active vertex, which bounds its degree in $G[\I]$ to $k$, given that out of its potential $k$ neighbors in $G_L$, one of them, namely $s$, is not inserted into $\I$. All the vertices that are inserted into $G_L$ at the last step of the iteration, other than $w$ and $z$, have degree $3$ in $G$, and hence at most $3$ in $G[\I]$. The vertex between $w$ and $z$ that is not selected as the special vertex has two neighbors in $G_R$ and at most two neighbors in $G_L$, however one of the two neighbors in $G_L$ is $h$, hence that vertex has degree at most $3\leq k$ in $G[\I]$. Finally, the vertex between $w$ and $z$ that is selected as the special vertex has at most $k-1$ neighbors from $G_R$ in $\I$ (this is where we use the condition on the degree of the special vertex), and at most two neighbors in $G_L$, however one of the two neighbors in $G_L$ is $h$, hence that vertex has degree at most $k$ in $G[\I]$. 
\end{itemize}
This concludes the proof of the theorem.
\end{proof}

We conclude the section by showing lower bounds for the size of an induced subgraph of degree at most $k$ that one is guaranteed to find in an outerplanar graph. We exhibit three lower bounds, one for $k=3$, one for $k\geq 4$ with $k$ even, and one for $k\geq 5$ with $k$ odd. The first lower bound is tight, because of \cref{th:outerplanar-ub}. 
%
%
Before stating the theorems, we present some tools that we are going to use for their proofs. By the same reasoning used for outerpaths, we can assume that the input is a maximal outerplanar graph $G$. Our first tool is a decomposition lemma, which is stated in the following.

\begin{lemma}\label{le:decompose-outerplanar}
Let $G$ be an $n$-vertex maximal outerplanar graph, let $\ell\geq 1$ be an integer, and let $xy$ be an edge incident to the outer face of $\mathcal O_G$. If $n\geq \ell+2$, there exists an edge $uv\in E(G)$ such that a $uv$-split subgraph $H$ of $G$ satisfies the following properties:
\begin{itemize}
\item $H$ has a number $h$ of vertices such that $\ell+2\leq h\leq 2\ell+1$; and 
\item if $uv$ does not coincide with the edge $xy$, then $H$ does not contain the edge $xy$.
\end{itemize}
\end{lemma}
\begin{proof}
The proof is by induction on $n$. In the base case we have $\ell+2\leq n\leq 2\ell+1$. Then let $uv$ coincide with $xy$, thus satisfying the second property. Also, $uv$ splits $G$ into two outerplanar graphs, one of which is $G$ and has hence the required number of vertices.

If $n\geq 2\ell+2$, let $z$ be the vertex such that the cycle $(x,y,z)$ bounds an internal face of~$\mathcal O_G$. Let~$G_1$ be the $xz$-split subgraph of $G$ that does not contain~$y$. Let~$G_2$ be the $yz$-split subgraph of $G$ that does not contain~$x$. We distinguish two cases.

\begin{itemize}
\item If both $G_1$ and $G_2$ have less than $2\ell+2$ vertices, then  at least one of them has at least $\ell+2$ vertices, given that $|V(G_1)|+|V(G_2)|=n+1\geq 2\ell+3$, where the first equality is due to the fact that $G_1$ and $G_2$ share~$z$. Then $xz$ or $yz$ is the required edge, depending on whether $G_1$ or $G_2$ has at least $\ell+2$ vertices, respectively. 
\begin{figure}[htb]
	\centering
	\includegraphics[scale=.8, page=1]{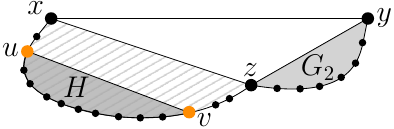}
	\caption{Illustration for the proof of~\cref{le:decompose-outerplanar}. The case in which $G_1$ has at least $2\ell+2$ vertices. Most internal edges of $G$ are not shown, the interior of $G_1$ is striped.}
	\label{fig:lower-bound-outerplanar-lemma}
\end{figure}
\item Otherwise, one of $G_1$ and $G_2$, say $G_1$, has at least $2\ell+2$ vertices. See \cref{fig:lower-bound-outerplanar-lemma}. By induction, $G_1$ contains an edge $uv$ such that a $uv$-split subgraph $H$ of $G_1$ has a number~$h$ of vertices such that $\ell+2\leq h\leq 2\ell+1$ and does not contain the edge $xz$. The latter condition implies that $H$ is also one of the two $uv$-split subgraphs of $G$ and that $H$ does not contain the edge $xy$, as required.
\end{itemize}
This concludes the proof of the lemma.
\end{proof}

For a subgraph $H$ of $G$, we denote by $\overline{H}$ the subgraph of $G$ induced by $V(G)\setminus V(H)$. Note that $\overline{H}$ is a (not necessarily maximal) outerplanar graph. The basic strategy we use in order to find a large set $\mathcal I \subseteq V(G)$ of vertices such that $G[\mathcal I]$ has degree at most $k$ is the following. By employing \cref{le:decompose-outerplanar}, we find an edge $uv$ such that, in a $uv$-split subgraph $H$ of $G$, we can select a set~$\mathcal I_H\subset V(H)$ of vertices that contains the fraction of the vertices in $V(H)$ we are aiming for ($\frac{2}{3}$ for $k=3$, $\frac{2k+1}{2k+5}$ for $k \geq 4$ with $k$ even, and $\frac{2k+2/3}{2k+14/3}$ for $k \geq 5$ with $k$ odd), that induces a subgraph of $G$ with degree at most $k$, and that does not contain $u$ and $v$. We also apply induction to find a large set of vertices $\mathcal I_{\overline{H}}$ in $\overline{H}$ such that $G[\mathcal I_{\overline{H}}]$ has degree at most $k$. Then the set $\mathcal I=\mathcal I_H \cup \mathcal I_{\overline{H}}$ has the required number of vertices and $G[\mathcal I]$ has degree at most~$k$. This is formalized in the following lemma.

\begin{lemma}\label{le:tool-degree-k}
Let $k\geq 3$ be an integer and let $0<c<1$ be a real number. Suppose that $G$ contains an edge $uv$ such that one of the two $uv$-split subgraphs of $G$, say $H$, has a set~$\mathcal I_H$ of vertices such that $\mathcal I_H$ contains neither $u$ nor $v$, such that $|\mathcal I_H|\geq c\cdot |V(H)|$, and such that $G[\mathcal I_H]$ has degree at most $k$. Suppose also that $\overline{H}$ has a set $\mathcal I_{\overline{H}}$ of vertices such that $|\mathcal I_{\overline{H}}|\geq c\cdot |V(\overline{H})|$ and such that $G[\mathcal I_{\overline{H}}]$ has degree at most $k$.
	
Then the set $\mathcal I:=\mathcal I_H\cup \mathcal I_{\overline{H}}$ contains at least $c\cdot n$ vertices and $G[\mathcal I]$ has degree at most $k$.
\end{lemma}
\begin{proof}
	That $\mathcal I$ contains at least $c \cdot n$ vertices follows by the assumption on the cardinalities of $\mathcal I_H$ and $\mathcal I_{\overline{H}}$, and by $|V(H)|+|V(\overline{H})|=n$. That $G[\mathcal I]$ has degree at most $k$ follows by the assumption on the degrees of $G[\mathcal I_H]$ and $G[\mathcal I_{\overline{H}}]$, and by the fact that $u$ and $v$, which do not belong to $\mathcal I_H$, are the only vertices of $H$ with neighbors in $\overline{H}$. 
\end{proof}

We are now ready to present our theorems. We start with the one for induced subgraphs of degree at most $3$.

\begin{theorem} \label{th:outerplanar-lb-3}
	For every $n\geq 1$, for every $n$-vertex outerplanar graph $G$, there exists a set $\mathcal I \subseteq V(G)$ such that $G[\mathcal I]$ has degree at most $3$ and $|\mathcal I|\geq \frac{2}{3}n$.
\end{theorem}


\begin{proof}
Let $G$ be an $n$-vertex maximal outerplanar graph. The proof is by induction on $n$. In the base case we have $n\leq 11$. An easy edge-counting argument proves that we can find a set $\mathcal I$ with $n$ vertices (if $1\leq n\leq 4$), with $n-1$ vertices (if $5\leq n\leq 6$), with $n-2$ vertices (if $7\leq n\leq 8$), or with $n-3$ vertices (if $9\leq n\leq 10$) such that $G[\mathcal I]$ has degree at most $3$. In all four cases the ratio $\frac{|\mathcal I|}{n}$ is larger than or equal to $\frac{2}{3}$.  The edge-counting argument is as follows. Observe that $\mathcal O_G$ has $n-3$ internal edges and that, for any set $\mathcal I\subseteq V(G)$, a vertex with degree at least $4$ in $G[\mathcal I]$ has at least two incident edges that are internal in $\mathcal O_G$. Hence, we can initialize $\mathcal I=V(G)$ and then repeatedly remove from $\mathcal I$ a vertex with degree at least~$4$ in $G[\mathcal I]$, while such a vertex exists. If $n\leq 4$, then $G$ has at most $1$ edge that is internal in $\mathcal O_G$, hence no vertex is removed from $\mathcal I$. If $5\leq n\leq 6$, then $G$ has up to $3$ edges that are internal in $\mathcal O_G$, hence up to $1$ vertex is removed from $\mathcal I$. If $7\leq n\leq 8$, then $G$ has up to $5$ edges that are internal in $\mathcal O_G$, hence up to $2$ vertices are removed from $\mathcal I$. Finally, if $9\leq n\leq 10$, then $G$ has up to $7$ edges that are internal in $\mathcal O_G$, hence up to $3$ vertices are removed from $\mathcal I$. 

If $n=11$, the above argument requires some amendment, since $G$ has $8$ internal edges, hence charging pairs of internal edges to removed degree-$4$ vertices only upper bounds the number of removed vertices by $4$, while we need to show that the number of removed vertices is at most $3$ in order to get our desired~$\frac{2}{3}$ ratio. The key observation is that, after initializing~$\mathcal I=V(G)$ and removing from $\mathcal I$ the first vertex $v$ with degree at least~$4$ in $G$, the number of internal edges in $G[\mathcal I]$ decreases at least by three. This is obvious if the degree of $v$ in $G$ is at least $5$. If the degree of $v$ in $G$ is $4$, let $e_1$ and $e_2$ be the internal edges of $G$ incident to $v$, and let $v_0$, $v_1$, $v_2$, and $v_3$ be the neighbors of $v$, where $e_1=vv_1$ and $e_2=vv_2$ lie inside the cycle~$(v,v_0,v_1,v_2,v_3)$ of $G$. Then one of the edges of~ the path $(v_0,v_1,v_2,v_3)$ is an internal edge~$e$ of~$G$, as otherwise $G$ would not be a maximal outerplanar graph or it would contain $5$ vertices. Hence, the removal of $v$ from $\mathcal I$ not only removes $e_1$ and $e_2$ from~$G[\mathcal I]$, but also makes~$e$ an external edge of $G[\mathcal I]$. Since the number of internal edges of $G[\mathcal I]$ after the removal of $v$ is at most $5$, at most $2$ more degree-$4$ vertices are then removed from~$G[\mathcal I]$.

Suppose now that $n\geq 12$. As argued before the theorem, it suffices to discuss the case in which $G$ is a maximal outerplanar graph. We define a $uv$-split subgraph $H$ of $G$ with a suitable size, then we select a large set $\mathcal I_H\subseteq V(H)$ of vertices, with $u,v \notin \mathcal I_H$, such that $G[\mathcal I_H]$ has degree at most $3$. Applying induction on $\overline{H}$ to get a set $\mathcal I_{\overline H}$ and then applying \cref{le:tool-degree-k} allows us to define the desired set $\mathcal I$. We distinguish $16$ cases, labeled according to the size of a $uv$-split subgraph that we can find in $G$ -- we need to consider every size in this range. Eventually, we will use \cref{le:decompose-outerplanar} with $\ell=10$, in order to find a $uv$-split subgraph with a number of vertices between $12$ and~$21$, however for our arguments we need to consider $uv$-split subgraphs of smaller sizes, as well. Specifically, if we can find a~$uv$-split subgraph~$H$ with a number of vertices that is~$6$,~$9$,~$10$, or between~$12$ and~$21$, we can select a large set~$\mathcal I_H\subseteq V(H)$ of vertices, with~$u,v \notin \mathcal I_H$, such that~$G[\mathcal I_H]$ has degree at most~$3$, which allows us to apply \cref{le:tool-degree-k} to find the desired set~$\mathcal I$ of vertices in~$G$. If~$H$ has~$7$,~$8$, or~$11$ vertices, we can still select a large set~$\mathcal I_H\subseteq V(H)$ of vertices such that~$G[\mathcal I_H]$ has degree at most~$3$, however this requires~$u$ and/or~$v$ to be part of~$\mathcal I_H$. The fact that~$u$ and/or~$v$ belong to~$\mathcal I_H$ prevents us from employing \cref{le:tool-degree-k} to find the desired set of vertices in~$G$, however the existence of such a set~$\mathcal I_H$ is used when~$uv$-split subgraphs with more vertices are considered. In the following, whenever we consider a $uv$-split subgraph $H$ of $G$,  we denote by $w$ the vertex such that the cycle $(u,v,w)$ bounds an internal face of $\mathcal O_H$ and by $h$ the number of vertices of $H$. Also, we denote by $H_1$ the $uw$-split subgraph of $G$ not containing $v$ and by $H_2$ the $vw$-split subgraph of $G$ not containing $u$. We let $h_1$ and $h_2$ be the number of vertices of $H_1$ and $H_2$, respectively, where $h_1\geq 2$, $h_2\geq 2$,  and $h_1+h_2=h+1$, given that $w$ belongs both to $H_1$ and to $H_2$. Finally, we denote by $\mathcal I_{\overline{H}} \subseteq V({\overline{H}})$ a set such that $G[\mathcal I_{\overline{H}}]$ has degree at most $3$ and $|\mathcal I_{\overline{H}}|\geq \frac{2}{3}|V({\overline{H}})|$. This set exists by induction.

{\bf Case h=6:} Suppose that $G$ contains an edge $uv$ such that a $uv$-split subgraph $H$ of $G$ has $6$ vertices. We define $\mathcal I_H:=V(H)\setminus\{u,v\}$. Note that $|\mathcal I_H|/|V(H)|=\frac{4}{6}=\frac{2}{3}$. Also, $\mathcal I_H$ contains $4$ vertices, hence $G[\mathcal I_H]$ has degree at most~$3$. By \cref{le:tool-degree-k}, the set $\mathcal I:=\mathcal I_H\cup \mathcal I_{\overline{H}}$ contains at least $\frac{2}{3}n$ vertices and $G[\mathcal I]$ has degree at most~$3$, as desired. We henceforth assume that $G$ has no edge $uv$ such that a $uv$-split subgraph of $G$ has $6$ vertices. 


{\bf Case h=7:}  Suppose that $G$ contains an edge $uv$ such that a $uv$-split subgraph $H$ of $G$ has $h=7$ vertices. Then we have $h_1\geq 3$, as otherwise $H_2$ would have $6$ vertices. Symmetrically, $h_2\geq 3$. If there exists a vertex in $V(H)\setminus \{w\}$ that is not adjacent to $w$, then we define $\mathcal I_H:=V(H)\setminus\{u,v\}$. Then $|\mathcal I_H|/|V(H)|=\frac{5}{7}>\frac{2}{3}$. Also, $G[\mathcal I_H]$ has degree at most $3$, given that $w$ is adjacent to at most $3$ vertices in $\mathcal I_H$, by assumption, given that each vertex in $V(H_1)\setminus \{u,w\}$ is not adjacent to at least $1$ vertex of $H_2$, and given that each vertex in $V(H_2)\setminus \{v,w\}$ is not adjacent to at least $1$ vertex of $H_1$. By \cref{le:tool-degree-k}, the set $\mathcal I:=\mathcal I_H\cup \mathcal I_{\overline{H}}$ contains at least $\frac{2}{3}n$ vertices and $G[\mathcal I]$ has degree at most $3$, as desired. 

We henceforth assume that, if $G$ has an edge $uv$ such that a $uv$-split subgraph $H$ of $G$ has $7$ vertices, then the degree of $w$ in $H$ is $6$. This implies that $H$ admits a \emph{$5$-over-$7$ set with $u$}, a \emph{$5$-over-$7$ set with $v$}, and a \emph{$6$-over-$7$ set with $u$ and $v$}, which are defined as follows. A $5$-over-$7$ set with $u$ is a set $\mathcal I_H$ of vertices such that $|\mathcal I_H|=5$, such that $u\in \mathcal I_H$ and $v\notin \mathcal I_H$, such that $G[\mathcal I_H]$ has degree at most $3$, and such that the degree of $u$ in $G[\mathcal I_H]$ is at most $1$. Indeed, the set $\mathcal I_H:=V(H)\setminus \{v,w\}$ satisfies these properties, see \cref{fig:outerplanar-lb-7}(a). The definition of a $5$-over-$7$ set with $v$ is analogous, with $v$ in place of $u$ and vice versa, and is realized by the set $\mathcal I_H:=V(H)\setminus \{u,w\}$, see \cref{fig:outerplanar-lb-7}(b). A $6$-over-$7$ set with $u$ and $v$ is a set $\mathcal I_H$ of vertices such that $|\mathcal I_H|=6$, such that $u,v\in \mathcal I_H$, such that $G[\mathcal I_H]$ has degree at most $3$, and such that the degree of $u$ and $v$ in $G[\mathcal I_H]$ is at most $2$. Indeed, the set $\mathcal I_H:=V(H)\setminus \{w\}$ satisfies these properties, see \cref{fig:outerplanar-lb-7}(c). Although we cannot use such sets to exclude the existence of a $uv$-split subgraph with $7$ vertices, since they contain $u$ or $v$, we will use them in later cases.

\begin{figure}[htb]
	\centering
	\begin{subfigure}[b]{0.3\textwidth}
		\centering
		\includegraphics[scale=.7, page=7]{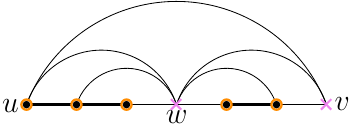}
		\subcaption{}
	\end{subfigure}
	\hfill
	\begin{subfigure}[b]{0.3\textwidth}
		\centering
		\includegraphics[scale=.7, page=8]{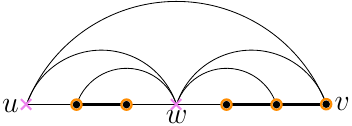}
		\subcaption{}
	\end{subfigure}
	\hfill
	\begin{subfigure}[b]{0.3\textwidth}
		\centering
		\includegraphics[scale=.7, page=9]{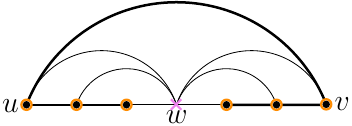}
		\subcaption{}
	\end{subfigure}
	\caption{A $uv$-split subgraph $H$ with $7$ vertices in which the vertex $w$ has degree $6$. (a) A $5$-over-$7$ set with~$u$. (b) A $5$-over-$7$ set with $v$.  (c) A $6$-over-$7$ set with $u$ and $v$. Vertices in the sets have a circular contour, while vertices not in the sets have a cross.}
	\label{fig:outerplanar-lb-7}
\end{figure}


{\bf Case h=8:}  Suppose that $G$ contains an edge $uv$ such that a $uv$-split subgraph~$H$ of~$G$ has $h=8$ vertices. Since $h_1+h_2=h+1$ and $G$ has no edge $uv$ such that a $uv$-split subgraph of $G$ has $6$ vertices, we have that either one of $h_1$ and $h_2$ is $2$ and the other one is $7$, or one of $h_1$ and $h_2$ is $4$ and the other one is $5$. 
\begin{itemize}
	\item Suppose that $h_1=2$ and $h_2=7$, the case in which $h_1=7$ and $h_2=2$ is symmetric. By assumption, $H_2$ admits a $5$-over-$7$ set with $w$, say $\mathcal I^w_{H_2}$, and a $6$-over-$7$ set with $w$ and~$v$, say $\mathcal I^{wv}_{H_2}$. This implies that $H$ admits a \emph{$6$-over-$8$ set with $u$} and a \emph{$6$-over-$8$ set with $v$}. The former is a set $\mathcal I_H$ of vertices such that $|\mathcal I_H|=6$, such that $u\in \mathcal I_H$ and $v\notin \mathcal I_H$, such that $G[\mathcal I_H]$ has degree at most $3$, and such that the degree of $u$ in $G[\mathcal I_H]$ is at most~$2$. Indeed, the set $\mathcal I^w_{H_2}\cup \{u\}$ satisfies these properties, see \cref{fig:outerplanar-lb-8-1}(a). The definition of a $6$-over-$8$ set with $v$ is analogous, with $v$ in place of $u$ and vice versa, and is realized by the set $\mathcal I^{wv}_{H_2}$, see \cref{fig:outerplanar-lb-8-1}(b). We also have that $H$ admits a \emph{$6$-over-$8$ set with $u$-deg-$1$} or a \emph{$6$-over-$8$ set with $v$-deg-$1$}. The former is a $6$-over-$8$ set with $u$ such that the degree of $u$ in $G[\mathcal I_H]$ is at most~$1$. The latter is a $6$-over-$8$ set with $v$ such that the degree of $v$ in $G[\mathcal I_H]$ is at most~$1$. Indeed, with the current assumption $h_1=2$ and $h_2=7$, we have that $H$ admits a $6$-over-$8$ set with $u$-deg-$1$, while if $h_1=7$ and $h_2=2$, then $H$ admits a $6$-over-$8$ set with $v$-deg-$1$. Although we cannot use such sets to exclude the existence of a $uv$-split subgraph with $8$ vertices, given that they contain $u$ or $v$, we will use them in later cases.

\begin{figure}[htb]
	\centering
	\begin{subfigure}[b]{0.45\textwidth}
		\centering
		\includegraphics[scale=.7, page=10]{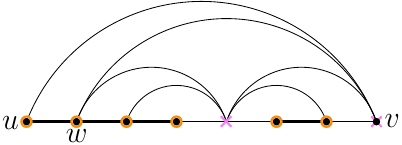}
		\subcaption{}
	\end{subfigure}
	\hfill
	\begin{subfigure}[b]{0.45\textwidth}
		\centering
		\includegraphics[scale=.7, page=11]{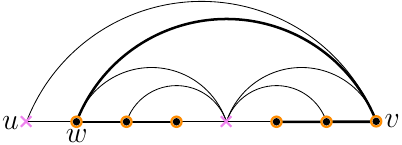}
		\subcaption{}
	\end{subfigure}
	\caption{A $uv$-split subgraph $H$ with $8$ vertices with $h_1=2$ and $h_2=7$. (a) A $6$-over-$8$ set with~$u$. (b) A $6$-over-$8$ set with $v$.}
	\label{fig:outerplanar-lb-8-1}
\end{figure}

	\item Suppose that $h_1=4$ and $h_2=5$, the case in which $h_1=5$ and $h_2=4$ is symmetric. If $w$ has at most $3$ neighbors among the vertices in $V(H)\setminus\{u,v\}$, then we can define $\mathcal I_H:=V(H)\setminus\{u,v\}$. Then $|\mathcal I_H|/|V(H)|=\frac{6}{8}>\frac{2}{3}$. Also, $G[\mathcal I_H]$ has degree at most $3$. Indeed, consider any vertex $z\in \mathcal I_H$. If $z=w$, then $z$ is adjacent to at most $3$ vertices in~$\mathcal I_H$, by assumption. If $z\in V(H_2)$ and $z\neq w$ then $z$ is not adjacent to the $2$ vertices of~$H_1$ different from $u$ and $w$ and is not adjacent to $z$ itself. Since $u$ and $v$ are not in~$\mathcal I_H$, we have that $z$ has at most $3$ neighbors in $G[\mathcal I_H]$. The argument for the case in which $z\in V(H_1)$ and $z\neq w$ is analogous.

\begin{figure}[htb]
	\centering
	\begin{subfigure}[b]{0.45\textwidth}
		\centering
		\includegraphics[scale=.7, page=12]{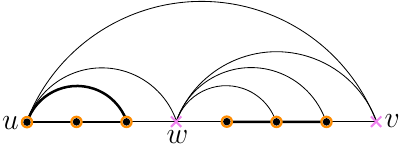}
		\subcaption{}
	\end{subfigure}
	\hfill
	\begin{subfigure}[b]{0.45\textwidth}
		\centering
		\includegraphics[scale=.7, page=13]{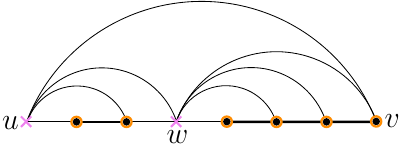}
		\subcaption{}
	\end{subfigure}
	\caption{A $uv$-split subgraph $H$ with $8$ vertices with $h_1=4$ and $h_2=5$. (a) A $6$-over-$8$ set with~$u$. (b) A $6$-over-$8$ set with $v$.}
	\label{fig:outerplanar-lb-8-2}
\end{figure}
	
	We henceforth assume that, if $G$ has an edge $uv$ such that a $uv$-split subgraph $H$ of~$G$ has $8$ vertices and if the $uw$-split subgraph $H_1$ has $4$ or $5$ vertices, then $w$ has degree at least $6$ in $H$. This again implies that $H$ admits a $6$-over-$8$ set with $u$ and a $6$-over-$8$ set with $v$. Indeed, the former is realized by the set $V(H)\setminus \{v,w\}$, see \cref{fig:outerplanar-lb-8-2}(a), and the latter is realized by the set $V(H)\setminus \{u,w\}$, see \cref{fig:outerplanar-lb-8-2}(b). Also, $H$ admits a $6$-over-$8$ set with $u$-deg-$1$ or a $6$-over-$8$ set with $v$-deg-$1$. Indeed, with the current assumption $h_1=4$ and $h_2=5$, we have that $H$ admits a $6$-over-$8$ set with $v$-deg-$1$, see again \cref{fig:outerplanar-lb-8-2}(b), while if $h_1=5$ and $h_2=4$ then $H$ admits a $6$-over-$8$ set with $u$-deg-$1$.
\end{itemize} 

{\bf Case h=9:}  Suppose that $G$ contains an edge $uv$ such that a $uv$-split subgraph $H$ of $G$ has $h=9$ vertices. Since $h_1+h_2=h+1$ and $G$ has no edge $xy$ such that an $xy$-split subgraph of $G$ has $6$ vertices, we have that either one of $h_1$ and $h_2$ is $2$ and the other one is $8$, or  one of $h_1$ and $h_2$ is $3$ and the other one is $7$, or both $h_1$ and $h_2$ are $5$. 

\begin{itemize}
	\item Suppose that $h_1=2$ and $h_2=8$, the case in which $h_1=8$ and $h_2=2$ is symmetric. By assumption, $H_2$ admits a $6$-over-$8$ set with $w$, say $\mathcal I^w_{H_2}$. Then we can define $\mathcal I_H:=\mathcal I^w_{H_2}$. Since $\mathcal I^w_{H_2}$ is a $6$-over-$8$ set with $w$, we have that $G[\mathcal I_H]$ has degree at most $3$ and $|\mathcal I_H|/|V(H)|=\frac{6}{9}=\frac{2}{3}$. Since $u$ and $v$ do not belong to $\mathcal I_H$, by \cref{le:tool-degree-k}, the set $\mathcal I:=\mathcal I_H\cup \mathcal I_{\overline{H}}$ contains at least $\frac{2}{3}n$ vertices and $G[\mathcal I]$ has degree at most~$3$, as desired. 
	\item Suppose next that $h_1=3$ and $h_2=7$, the case in which $h_1=7$ and $h_2=3$ is symmetric. By assumption, $H_2$ admits a $5$-over-$7$ set with $w$, say $\mathcal I^w_{H_2}$. We define $\mathcal I_H:=\mathcal I^w_{H_2} \cup \{t\}$, where $t$ is the vertex of $H_1$ different from $u$ and $w$. Note that $|\mathcal I_H|/|V(H)|=\frac{6}{9}=\frac{2}{3}$. Since $\mathcal I^w_{H_2}$ is a $5$-over-$7$ set with $w$ and since $w$ is the only vertex of $H_2$ adjacent to $t$, we have that $G[\mathcal I_H]$ has degree at most $3$; in particular, $w$ has degree at most $1$ in $G[\mathcal I^w_{H_2}]$, hence at most $2$ in $G[\mathcal I_H]$. Since $u$ and $v$ do not belong to $\mathcal I_H$, by \cref{le:tool-degree-k}, the set $\mathcal I:=\mathcal I_H\cup \mathcal I_{\overline{H}}$ contains at least $\frac{2}{3}n$ vertices and $G[\mathcal I]$ has degree at most~$3$, as desired. 
	\item Suppose finally that $h_1=5$ and $h_2=5$. We define $\mathcal I_H:=V(H)\setminus\{u,v,w\}$. Note that $|\mathcal I_H|/|V(H)|=\frac{6}{9}=\frac{2}{3}$. Also, each vertex $z$ in $\mathcal I_H\cap V(H_1)$ is adjacent to at most $2$ vertices in $\mathcal I_H$, namely those in $\mathcal I_H\cap V(H_1)$ different from $z$ itself, given that $u$ and $w$ are the only vertices of $H_1$ adjacent to vertices of $H_2$. Analogously, each vertex in $\mathcal I_H\cap V(H_2)$ is adjacent to at most $2$ vertices in $\mathcal I_H$. It follows that $G[\mathcal I_H]$ has degree at most~$2$. By \cref{le:tool-degree-k}, the set $\mathcal I:=\mathcal I_H\cup \mathcal I_{\overline{H}}$ contains at least $\frac{2}{3}n$ vertices and $G[\mathcal I]$ has degree at most~$3$, as desired. 
\end{itemize} 
We henceforth assume that $G$ has no edge $uv$ such that a $uv$-split subgraph of $G$ has~$9$~vertices. 

{\bf Case h=10:}  Suppose that $G$ contains an edge $uv$ such that a $uv$-split subgraph $H$ of $G$ has $h=10$ vertices. Since $h_1+h_2=h+1$ and $G$ has no edge $xy$ such that an $xy$-split subgraph of $G$ has $6$ or $9$ vertices, we have that either one of $h_1$ and $h_2$ is $3$ and the other one is $8$, or one of $h_1$ and $h_2$ is $4$ and the other one is $7$. 

\begin{itemize}
	\item Suppose that $h_1=3$ and $h_2=8$, the case in which $h_1=8$ and $h_2=3$ is symmetric. By assumption, $H_2$ admits a $6$-over-$8$ set with $w$, say $\mathcal I^w_{H_2}$. Then we can define $\mathcal I_H:=\mathcal I^w_{H_2}\cup \{t\}$, where $t$ is the vertex of $H_1$ different from $u$ and $w$. Note that $|\mathcal I_H|/|V(H)|=\frac{7}{10}>\frac{2}{3}$. Since $\mathcal I^w_{H_2}$ is a $6$-over-$8$ set with $w$ and since $w$ is the only vertex of $H_2$ adjacent to $t$, we have that $G[\mathcal I_H]$ has degree at most $3$; in particular, $w$ has degree at most $2$ in $G[\mathcal I^w_{H_2}]$, hence at most $3$ in $G[\mathcal I_H]$. Since $u$ and $v$ do not belong to $\mathcal I_H$, by \cref{le:tool-degree-k}, the set $\mathcal I:=\mathcal I_H\cup \mathcal I_{\overline{H}}$ contains at least $\frac{2}{3}n$ vertices and $G[\mathcal I]$ has degree at most~$3$, as desired. 
	\item Suppose that $h_1=4$ and $h_2=7$, the case in which $h_1=7$ and $h_2=4$ is symmetric. By assumption, $H_2$ admits a $5$-over-$7$ set with $w$, say $\mathcal I^w_{H_2}$. Then we can define $\mathcal I_H:=\mathcal I^w_{H_2}\cup \{t_1,t_2\}$, where $t_1$ and $t_2$ are the vertices of $H_1$ different from $u$ and $w$. Note that $|\mathcal I_H|/|V(H)|=\frac{7}{10}>\frac{2}{3}$. Since $\mathcal I^w_{H_2}$ is a $5$-over-$7$ set with $w$ and since $w$ is the only vertex of $H_2$ adjacent to $t_1$ and $t_2$, we have that $G[\mathcal I_H]$ has degree at most $3$; in particular, $w$ has degree at most $1$ in $G[\mathcal I^w_{H_2}]$, hence at most $3$ in $G[\mathcal I_H]$. Since $u$ and $v$ do not belong to $\mathcal I_H$, by \cref{le:tool-degree-k}, the set $\mathcal I:=\mathcal I_H\cup \mathcal I_{\overline{H}}$ contains at least $\frac{2}{3}n$ vertices and $G[\mathcal I]$ has degree at most~$3$, as desired. 
\end{itemize} 
We henceforth assume that $G$ has no edge $uv$ such that a $uv$-split subgraph~of~$G$~has~$10$~vertices.  

{\bf Case h=11:}  Suppose that $G$ contains an edge $uv$ such that a $uv$-split subgraph $H$ of $G$ has $h=11$ vertices. Since $h_1+h_2=h+1$ and $G$ has no edge $xy$ such that an $xy$-split subgraph of $G$ has $6$, $9$, or $10$ vertices, we have that either one of $h_1$ and $h_2$ is $4$ and the other one is $8$, or one of $h_1$ and $h_2$ is $5$ and the other one is $7$. 

\begin{itemize}
    \item Suppose that $h_1=4$ and $h_2=8$, the case in which $h_1=8$ and $h_2=4$ is symmetric. By assumption, $H_2$ admits a $6$-over-$8$ set with $w$, say $\mathcal I^w_{H_2}$, and a $6$-over-$8$ set with $v$, say $\mathcal I^{v}_{H_2}$. Also by assumption, we have that $\mathcal I^w_{H_2}$ is actually a $6$-over-$8$ set with $w$-deg-$1$ or that $\mathcal I^{v}_{H_2}$ is actually a $6$-over-$8$ set with $v$-deg-$1$. 
 
    Suppose first that $\mathcal I^w_{H_2}$ is a $6$-over-$8$ set with $w$-deg-$1$. Then we can define $\mathcal I_H:=\mathcal I^w_{H_2}\cup \{t_1,t_2\}$, where $t_1$ and $t_2$ are the vertices of $H_1$ different from $u$ and $w$. Note that $|\mathcal I_H|/|V(H)|=\frac{8}{11}>\frac{2}{3}$. Since $\mathcal I^w_{H_2}$ is a $6$-over-$8$ set with $w$-deg-$1$ and since $w$ is the only vertex of $H_2$ adjacent to $t_1$ and $t_2$, we have that $G[\mathcal I_H]$ has degree at most $3$; in particular, $w$ has degree at most $1$ in $G[\mathcal I^w_{H_2}]$, hence at most $3$ in $G[\mathcal I_H]$. Since $u$ and $v$ do not belong to $\mathcal I_H$, by \cref{le:tool-degree-k}, the set $\mathcal I:=\mathcal I_H\cup \mathcal I_{\overline{H}}$ contains at least $\frac{2}{3}n$ vertices and $G[\mathcal I]$ has degree at most~$3$, as desired. 
    
    Suppose next that $\mathcal I^{v}_{H_2}$ is a $6$-over-$8$ set with $v$-deg-$1$. If $w$ has at most one neighbor among $t_1$ and $t_2$, then we can define $\mathcal I_H:= \mathcal I^w_{H_2}\cup \{t_1,t_2\}$ and conclude the induction as above, hence we can assume that $w$ is adjacent to both $t_1$ and $t_2$. This implies that $u$ is a neighbor of at most one of $t_1$ and $t_2$. It follows that $H$ admits a \emph{$8$-over-$11$ set with $u$} and a \emph{$8$-over-$11$ set with $v$}. The former is a set $\mathcal I_H$ of vertices such that $|\mathcal I_H|=8$, such that $u\in \mathcal I_H$ and $v\notin \mathcal I_H$, such that $G[\mathcal I_H]$ has degree at most $3$, and such that the degree of $u$ in $G[\mathcal I_H]$ is at most~$1$. Indeed, the set $(\mathcal I^w_{H_2}\setminus \{w\})\cup \{u,t_1,t_2\}$ satisfies these properties. The definition of a $8$-over-$11$ set with $v$ is analogous, with $v$ in place of $u$ and vice versa, and is realized by the set $\mathcal I^{v}_{H_2}\cup\{t_1,t_2\}$.


 \item Suppose that $h_1=5$ and $h_2=7$, the case in which $h_1=7$ and $h_2=5$ is symmetric. By assumption, $H_2$ admits a $5$-over-$7$ set with $w$, say $\mathcal I^w_{H_2}$, and a $5$-over-$7$ set with $v$, say $\mathcal I^v_{H_2}$.
 
 Suppose first that $w$ is adjacent to at most $2$ out of the $3$ vertices $t_1$, $t_2$, and $t_3$ of $H_1$ different from $u$ and $w$. Then we can define $\mathcal I_H:= \mathcal I^w_{H_2}\cup \{t_1,t_2,t_3\}$. Note that $|\mathcal I_H|/|V(H)|=\frac{8}{11}>\frac{2}{3}$. Since $\mathcal I^w_{H_2}$ is a $5$-over-$7$ set with $w$ and since $w$ is the only vertex of $H_2$ adjacent to $t_1$, $t_2$ or $t_3$, we have that $G[\mathcal I_H]$ has degree at most $3$; in particular, $w$ has degree at most $1$ in $G[\mathcal I^w_{H_2}]$, hence at most $3$ in $G[\mathcal I_H]$, by the assumption that it is not adjacent to all of $t_1$, $t_2$, and $t_3$. Since $u$ and $v$ do not belong to $\mathcal I_H$, by \cref{le:tool-degree-k}, the set $\mathcal I:=\mathcal I_H\cup \mathcal I_{\overline{H}}$ contains at least $\frac{2}{3}n$ vertices and $G[\mathcal I]$ has degree at most~$3$, as desired. 
 
 Suppose next that $w$ is adjacent to all of $t_1$, $t_2$, and $t_3$. This implies that $u$ is a neighbor of at most one of $t_1$, $t_2$, and $t_3$. It follows that $H$ admits a $8$-over-$11$ set with $u$ and a $8$-over-$11$ set with $v$. The former is realized by the set $(\mathcal I^w_{H_2}\setminus \{w\})\cup \{u,t_1,t_2,t_3\}$ and the latter by the set $\mathcal I^{v}_{H_2}\cup\{t_1,t_2,t_3\}$. 
\end{itemize} 

Although we cannot exclude the existence of a $uv$-split subgraph with $11$ vertices, we will use the existence of a $8$-over-$11$ set with $u$ and of a $8$-over-$11$ set with $v$ in later cases.


{\bf Case h=12:}  Suppose that $G$ contains an edge $uv$ such that a $uv$-split subgraph $H$ of $G$ has $h=12$ vertices. Since $h_1+h_2=h+1$ and $G$ has no edge $xy$ such that an $xy$-split subgraph of $G$ has $6$, $9$, or $10$ vertices, we have that either one of $h_1$ and $h_2$ is $2$ and the other one is $11$, or one of $h_1$ and $h_2$ is $5$ and the other one is $8$. 

\begin{itemize}
    \item Suppose that $h_1=2$ and $h_2=11$, the case in which $h_1=11$ and $h_2=2$ is symmetric. By assumption, $H_2$ admits a $8$-over-$11$ set with $w$, say $\mathcal I^w_{H_2}$. Then we can define $\mathcal I_H:= \mathcal I^w_{H_2}$. Note that $|\mathcal I_H|/|V(H)|=\frac{8}{12}=\frac{2}{3}$. Since $\mathcal I^w_{H_2}$ is a $8$-over-$11$ set with $w$, we have that $G[\mathcal I_H]$ has degree at most $3$. Since $u$ and $v$ do not belong to $\mathcal I_H$, by \cref{le:tool-degree-k}, the set $\mathcal I:=\mathcal I_H\cup \mathcal I_{\overline{H}}$ contains at least $\frac{2}{3}n$ vertices and $G[\mathcal I]$ has degree at most~$3$, as desired. 
    
    \item Suppose that $h_1=5$ and $h_2=8$, the case in which $h_1=8$ and $h_2=5$ is symmetric. By assumption, $H_2$ admits a $6$-over-$8$ set with $w$, say $\mathcal I^w_{H_2}$. Let $\mathcal I_{H_2}$ be the set $\mathcal I^w_{H_2}\setminus \{w\}$. Then we can define $\mathcal I_H:=\mathcal I_{H_2}\cup \{t_1,t_2,t_3\}$, where $t_1$, $t_2$, and $t_3$ are the vertices of $H_1$ different from $u$ and $w$. Note that $|\mathcal I_{H_2}|=5$, hence $|\mathcal I_H|/|V(H)|=\frac{8}{12}=\frac{2}{3}$. Since $u$ and $w$ are the only vertices of $H_1$ with neighbors in $H_2$, each of $t_1$, $t_2$, and $t_3$ has degree at most $2$ in $G[\mathcal I_H]$. Since $\mathcal I^w_{H_2}$ is a $6$-over-$8$ set, we have that $G[\mathcal I^w_{H_2}]$ and then also $G[\mathcal I_{H_2}]$ has degree at most $3$. Since $v$ and $w$ are the only vertices of $H_2$ with neighbors in $H_1$, it follows that every vertex in $\mathcal I_{H_2}$ has degree at most $3$ in $G[\mathcal I_H]$. Since $u$ and $v$ do not belong to $\mathcal I_H$, by \cref{le:tool-degree-k}, the set $\mathcal I:=\mathcal I_H\cup \mathcal I_{\overline{H}}$ contains at least $\frac{2}{3}n$ vertices and $G[\mathcal I]$ has degree at most~$3$, as desired. 
\end{itemize}

We henceforth assume that $G$ has no edge $uv$ such that a $uv$-split subgraph of $G$ has $12$ vertices. 

{\bf Case h=13:}  Suppose that $G$ contains an edge $uv$ such that a $uv$-split subgraph $H$ of $G$ has $h=13$ vertices. Since $h_1+h_2=h+1$ and $G$ has no edge $xy$ such that an $xy$-split subgraph of $G$ has $6$, $9$, $10$, or $12$ vertices, we have that either one of $h_1$ and $h_2$ is $3$ and the other one is $11$, or $h_1=h_2=7$. 

\begin{itemize}
    \item Suppose that $h_1=3$ and $h_2=11$, the case in which $h_1=11$ and $h_2=3$ is symmetric. By assumption, $H_2$ admits a $8$-over-$11$ set with $w$, say $\mathcal I^w_{H_2}$. Then we can define $\mathcal I_H:=\mathcal I^w_{H_2}\cup \{t_1\}$, where $t_1$ is the vertex of $H_1$ different from $u$ and $w$. Note that $|\mathcal I_H|/|V(H)|=\frac{9}{13}>\frac{2}{3}$. Since $\mathcal I^w_{H_2}$ is a $8$-over-$11$ set with $w$, we have that $G[\mathcal I_H]$ has degree at most $3$; in particular, $w$ has at most one neighbor in $\mathcal I^w_{H_2}$, hence it has degree at most $2$ in $G[\mathcal I_H]$. Since $u$ and $v$ do not belong to $\mathcal I_H$, by \cref{le:tool-degree-k}, the set $\mathcal I:=\mathcal I_H\cup \mathcal I_{\overline{H}}$ contains at least $\frac{2}{3}n$ vertices and $G[\mathcal I]$ has degree at most~$3$, as desired. 
    
    \item Suppose that $h_1=h_2=7$. By assumption, $H_1$ admits a $5$-over-$7$ set with $w$, say $\mathcal I^w_{H_1}$, and $H_2$ admits a $5$-over-$7$ set with $w$, say $\mathcal I^w_{H_2}$. Then we can define $\mathcal I_H:=\mathcal I^w_{H_1}\cup \mathcal I^w_{H_2}$. Note that $|\mathcal I_H|/|V(H)|=\frac{9}{13}>\frac{2}{3}$. Since $\mathcal I^w_{H_1}$ and $\mathcal I^w_{H_2}$ are $5$-over-$7$ sets with $w$, since $u$ and $w$ are the only vertices of $H_1$ with neighbors in $H_2$, and since $v$ and $w$ are the only vertices of $H_2$ with neighbors in $H_1$, we have that $G[\mathcal I_H]$ has degree at most $3$; in particular, $w$ has degree at most $1$ in each of $G[\mathcal I^w_{H_1}]$ and $G[\mathcal I^w_{H_2}]$, hence at most $2$ in $G[\mathcal I_H]$. Since $u$ and $v$ do not belong to $\mathcal I_H$, by \cref{le:tool-degree-k}, the set $\mathcal I:=\mathcal I_H\cup \mathcal I_{\overline{H}}$ contains at least $\frac{2}{3}n$ vertices and $G[\mathcal I]$ has degree at most~$3$, as desired. 
\end{itemize}

 We henceforth assume that $G$ has no edge $uv$ such that a $uv$-split subgraph of $G$ has $13$ vertices. 

{\bf Case h=14:}  Suppose that $G$ contains an edge $uv$ such that a $uv$-split subgraph $H$ of $G$ has $h=14$ vertices. Since $h_1+h_2=h+1$ and $G$ has no edge $xy$ such that an $xy$-split subgraph of $G$ has $6$, $9$, $10$, $12$, or $13$ vertices, we have that either one of $h_1$ and $h_2$ is $4$ and the other one is $11$, or one of $h_1$ and $h_2$ is $7$ and the other one is $8$. 

\begin{itemize}
    \item Suppose that $h_1=4$ and $h_2=11$, the case in which $h_1=11$ and $h_2=4$ is symmetric. By assumption, $H_2$ admits a $8$-over-$11$ set with $w$, say $\mathcal I^w_{H_2}$. Then we can define $\mathcal I_H:=\mathcal I^w_{H_2}\cup \{t_1,t_2\}$, where $t_1$ and $t_2$ are the vertices of $H_1$ different from $u$ and $w$. Note that $|\mathcal I_H|/|V(H)|=\frac{10}{14}>\frac{2}{3}$. Since $\mathcal I^w_{H_2}$ is a $8$-over-$11$ set with $w$, we have that $G[\mathcal I_H]$ has degree at most $3$; in particular, $w$ has at most one neighbor in $\mathcal I^w_{H_2}$, hence it has degree at most $3$ in $G[\mathcal I_H]$. Since $u$ and $v$ do not belong to $\mathcal I_H$, by \cref{le:tool-degree-k}, the set $\mathcal I:=\mathcal I_H\cup \mathcal I_{\overline{H}}$ contains at least $\frac{2}{3}n$ vertices and $G[\mathcal I]$ has degree at most~$3$, as desired. 
    
    \item Suppose that $h_1=7$ and $h_2=8$, the other case is symmetric. By assumption, $H_1$ admits a $5$-over-$7$ set with $w$, say $\mathcal I^w_{H_1}$, and $H_2$ admits a $6$-over-$8$ set with $w$, say $\mathcal I^w_{H_2}$. Then we can define $\mathcal I_H:=\mathcal I^w_{H_1}\cup \mathcal I^w_{H_2}$. Note that $|\mathcal I_H|/|V(H)|=\frac{10}{14}>\frac{2}{3}$. Since $\mathcal I^w_{H_1}$ is a $5$-over-$7$ set with $w$, since $\mathcal I^w_{H_2}$ is a $6$-over-$8$ set with $w$, since $u$ and $w$ are the only vertices of $H_1$ with neighbors in $H_2$, and since $v$ and $w$ are the only vertices of $H_2$ with neighbors in $H_1$, we have that $G[\mathcal I_H]$ has degree at most $3$; in particular, $w$ has degree at most $1$ in $G[\mathcal I^w_{H_1}]$ and at most $2$ in $G[\mathcal I^w_{H_2}]$, hence at most $3$ in $G[\mathcal I_H]$. Since $u$ and $v$ do not belong to $\mathcal I_H$, by \cref{le:tool-degree-k}, the set $\mathcal I:=\mathcal I_H\cup \mathcal I_{\overline{H}}$ contains at least $\frac{2}{3}n$ vertices and $G[\mathcal I]$ has degree at most~$3$, as desired. 
\end{itemize}

We henceforth assume that $G$ has no edge $uv$ such that a $uv$-split subgraph of $G$ has $14$ vertices. 

{\bf Case h=15:}  Suppose that $G$ contains an edge $uv$ such that a $uv$-split subgraph $H$ of $G$ has $h=15$ vertices. Since $h_1+h_2=h+1$ and $G$ has no edge $xy$ such that an $xy$-split subgraph of $G$ has $6$, $9$, $10$, $12$, $13$, or $14$ vertices, we have that either one of $h_1$ and $h_2$ is $5$ and the other one is $11$, or $h_1=h_2=8$. 

\begin{itemize}
    \item Suppose that $h_1=5$ and $h_2=11$, the case in which $h_1=11$ and $h_2=5$ is symmetric. By assumption, $H_2$ admits a $8$-over-$11$ set with $w$, say $\mathcal I^w_{H_2}$. Then we can define $\mathcal I_H:=(\mathcal I^w_{H_2}\setminus\{w\})\cup \{t_1,t_2,t_3\}$, where $t_1$, $t_2$, and $t_3$ are the vertices of $H_1$ different from $u$ and $w$. Note that $|\mathcal I_H|/|V(H)|=\frac{10}{15}=\frac{2}{3}$. Since $\mathcal I^w_{H_2}$ is a $8$-over-$11$ set with $w$, we have that $G[\mathcal I_H]$ has degree at most $3$. Since $u$ and $v$ do not belong to $\mathcal I_H$, by \cref{le:tool-degree-k}, the set $\mathcal I:=\mathcal I_H\cup \mathcal I_{\overline{H}}$ contains at least $\frac{2}{3}n$ vertices and $G[\mathcal I]$ has degree at most~$3$, as desired. 
    
    \item Suppose that $h_1=h_2=8$. By assumption, $H_1$ admits a $6$-over-$8$ set with $w$, say $\mathcal I^w_{H_1}$, and $H_2$ admits a $6$-over-$8$ set with $w$, say $\mathcal I^w_{H_2}$. Let $\mathcal I_{H_1}$ be the set $\mathcal I^w_{H_1}\setminus \{w\}$ and let $\mathcal I_{H_2}$ be the set $\mathcal I^w_{H_2}\setminus \{w\}$. Then we can define $\mathcal I_H:=\mathcal I_{H_1}\cup \mathcal I_{H_2}$. Note that $|\mathcal I_H|/|V(H)|=\frac{10}{15}=\frac{2}{3}$. Since $u$ and $w$ are the only vertices of $H_1$ with neighbors in $H_2$, the degree in $G[\mathcal I_H]$ of each vertex in $\mathcal I_{H_1}$ is the same as its degree in $G[\mathcal I_{H_1}]$, hence at most $3$. Analogously,  the degree in $G[\mathcal I_H]$ of each vertex in $\mathcal I_{H_2}$ is at most $3$. Since $u$ and $v$ do not belong to $\mathcal I_H$, by \cref{le:tool-degree-k}, the set $\mathcal I:=\mathcal I_H\cup \mathcal I_{\overline{H}}$ contains at least $\frac{2}{3}n$ vertices and $G[\mathcal I]$ has degree at most~$3$, as desired.  
\end{itemize}

We henceforth assume that $G$ has no edge $uv$ such that a $uv$-split subgraph of $G$ has $15$ vertices. 

{\bf Case h=16:}  We observe that, with the current assumptions, $G$ does not contain any edge $uv$ such that a $uv$-split subgraph $H$ of $G$ has $h=16$ vertices. Indeed, since $h_1+h_2=h+1$, if $H$ had $h=16$ vertices, one of $H_1$ and $H_2$ would contain $6$, $9$, $10$, $12$, $13$, $14$, or $15$ vertices, which we already excluded.

{\bf Case h=17:}  Suppose that $G$ contains an edge $uv$ such that a $uv$-split subgraph $H$ of $G$ has $h=17$ vertices. Since $h_1+h_2=h+1$ and $G$ has no edge $xy$ such that an $xy$-split subgraph of $G$ has $9$, $10$, $12$, $13$, $14$, $15$, or $16$ vertices, we have that one of $h_1$ and $h_2$ is $7$ and the other one is $11$. Suppose that $h_1=7$ and $h_2=11$, as the other case is symmetric. By assumption, $H_1$ admits a $5$-over-$7$ set with $w$, say $\mathcal I^w_{H_1}$, and $H_2$ admits a $8$-over-$11$ set with $w$, say $\mathcal I^w_{H_2}$. Then we can define $\mathcal I_H:=\mathcal I^w_{H_1}\cup \mathcal I^w_{H_2}$. Note that $|\mathcal I_H|/|V(H)|=\frac{12}{17}>\frac{2}{3}$. Since $\mathcal I^w_{H_1}$ is a $5$-over-$7$ set with $w$, since $\mathcal I^w_{H_2}$ is a $8$-over-$11$ set with $w$, since $u$ and $w$ are the only vertices of $H_1$ with neighbors in $H_2$, and since $v$ and $w$ are the only vertices of $H_2$ with neighbors in $H_1$, we have that $G[\mathcal I_H]$ has degree at most $3$; in particular, $w$ has degree at most $1$ in each of $G[\mathcal I^w_{H_1}]$ and $G[\mathcal I^w_{H_2}]$, hence at most $2$ in $G[\mathcal I_H]$. Since $u$ and $v$ do not belong to $\mathcal I_H$, by \cref{le:tool-degree-k}, the set $\mathcal I:=\mathcal I_H\cup \mathcal I_{\overline{H}}$ contains at least $\frac{2}{3}n$ vertices and $G[\mathcal I]$ has degree at most~$3$, as desired. We henceforth assume that $G$ has no edge $uv$ such that a $uv$-split subgraph of $G$ has $17$ vertices. 

{\bf Case h=18:}  Suppose that $G$ contains an edge $uv$ such that a $uv$-split subgraph $H$ of $G$ has $h=18$ vertices. Since $h_1+h_2=h+1$ and $G$ has no edge $xy$ such that an $xy$-split subgraph of $G$ has $10$, $12$, $13$, $14$, $15$, $16$, or $17$ vertices, we have that one of $h_1$ and $h_2$ is $8$ and the other one is $11$. Suppose that $h_1=8$ and $h_2=11$, as the other case is symmetric. By assumption, $H_1$ admits a $6$-over-$8$ set with $w$, say $\mathcal I^w_{H_1}$, and $H_2$ admits a $8$-over-$11$ set with $w$, say $\mathcal I^w_{H_2}$. Then we can define $\mathcal I_H:=\mathcal I^w_{H_1}\cup \mathcal I^w_{H_2}$. Note that $|\mathcal I_H|/|V(H)|=\frac{13}{18}>\frac{2}{3}$. Since $\mathcal I^w_{H_1}$ is a $6$-over-$8$ set with $w$, since $\mathcal I^w_{H_2}$ is a $8$-over-$11$ set with $w$, since $u$ and $w$ are the only vertices of $H_1$ with neighbors in $H_2$, and since $v$ and $w$ are the only vertices of $H_2$ with neighbors in $H_1$, we have that $G[\mathcal I_H]$ has degree at most $3$; in particular, $w$ has degree at most $2$ in $G[\mathcal I^w_{H_1}]$ and at most $1$ in $G[\mathcal I^w_{H_2}]$, hence at most $3$ in $G[\mathcal I_H]$. Since $u$ and $v$ do not belong to $\mathcal I_H$, by \cref{le:tool-degree-k}, the set $\mathcal I:=\mathcal I_H\cup \mathcal I_{\overline{H}}$ contains at least $\frac{2}{3}n$ vertices and $G[\mathcal I]$ has degree at most~$3$, as desired. We henceforth assume that $G$ has no edge $uv$ such that a $uv$-split subgraph of $G$ has $18$ vertices. 

{\bf Case h=19:}  We observe that, with the current assumptions, $G$ does not contain any edge $uv$ such that a $uv$-split subgraph $H$ of $G$ has $h=19$ vertices. Indeed, since $h_1+h_2=h+1$, if $H$ had $h=19$ vertices, one of $H_1$ and $H_2$ would contain $9$, $10$, $12$, $13$, $14$, $15$, $16$, $17$, or $18$ vertices, which we already excluded.

{\bf Case h=20:}  We observe that, with the current assumptions, $G$ does not contain any edge $uv$ such that a $uv$-split subgraph $H$ of $G$ has $h=20$ vertices. Indeed, since $h_1+h_2=h+1$, if $H$ had $h=20$ vertices, one of $H_1$ and $H_2$ would contain $10$, $12$, $13$, $14$, $15$, $16$, $17$, $18$, or $19$ vertices, which we already excluded.

{\bf Case h=21:} By \cref{le:decompose-outerplanar}, applied with $\ell=10$, since $G$ has no edge $uv$ such that a $uv$-split subgraph of $G$ has a number of vertices between $12$ and $20$, we have that $G$ has an edge $uv$ such that a $uv$-split subgraph has $h=21$ vertices. Since $h_1+h_2=h+1$ and $G$ has no edge $xy$ such that an $xy$-split subgraph of $G$ has a number of vertices between $12$ and $20$, we have that $h_1$ and $h_2$ are both $11$. By assumption, $H_1$ admits a $8$-over-$11$ set with $w$, say $\mathcal I^w_{H_1}$, and $H_2$ admits a $8$-over-$11$ set with $w$, say $\mathcal I^w_{H_2}$. Then we can define $\mathcal I_H:=\mathcal I^w_{H_1}\cup \mathcal I^w_{H_2}$. Note that $|\mathcal I_H|/|V(H)|=\frac{15}{21}>\frac{2}{3}$. Since $\mathcal I^w_{H_1}$ and $\mathcal I^w_{H_2}$ are $8$-over-$11$ sets with $w$, it follows that the degree of $w$ in each of $G[\mathcal I^w_{H_1}]$ and $G[\mathcal I^w_{H_2}]$ is at most $1$, hence its degree in $G[\mathcal I_H]$ is at most $2$. Since the degree in $G[\mathcal I_H]$ of every vertex in $\mathcal I^w_{H_i}\setminus \{w\}$ is the same as its degree in $G[\mathcal I^w_{H_i}]$, for $i=1,2$, it follows that the degree of $G[\mathcal I_H]$ is at most $3$. Since $u$ and $v$ do not belong to $\mathcal I_H$, by \cref{le:tool-degree-k}, the set $\mathcal I:=\mathcal I_H\cup \mathcal I_{\overline{H}}$ contains at least $\frac{2}{3}n$ vertices and $G[\mathcal I]$ has degree at most~$3$, as desired. 
\end{proof}


We conclude with the theorem for induced subgraphs with degree at most $k\geq 4$.

\begin{theorem} \label{th:outerplanar-lb-k}
	For every $n\geq 1$, for every $n$-vertex outerplanar graph $G$, and for every integer $k\geq 4$, there exists a set $\mathcal I \subseteq V(G)$ such that $G[\mathcal I]$ has degree at most $k$ and $|\mathcal I|\geq \frac{2k+1}{2k+5}n$ (if $k$ is even) or $|\mathcal I|\geq \frac{2k+2/3}{2k+14/3}n$ (if $k$ is odd).
\end{theorem}

\begin{proof}
Let $G$ be an $n$-vertex maximal outerplanar graph and let $k$ be an integer greater than or equal to $4$. We assume first that $k$ is even, we will discuss later the case in which $k$ is odd. 

The proof is by induction on $n$. In the base case we have $n\leq \frac{3k}{2}+3$. If $n\leq k+1$, then we can set $\mathcal I := V(G)$ and obviously $G[\mathcal I]=G$ has degree at most $k$ and $|\mathcal I|> \frac{2k+1}{2k+5}n$. If $k+2\leq n \leq \frac{3k}{2}+3$, then the proof is similar to the one for the base case of \cref{th:outerplanar-lb-3}. Namely, we can initialize $\mathcal I=V(G)$ and then repeatedly remove from $\mathcal I$ a vertex with degree at least~$k+1$ in $G[\mathcal I]$, while such a vertex exists. The removal of the first vertex from~$\mathcal I$ removes at least $k$ internal edges from $G[\mathcal I]$ (by the same arguments presented for the case $n=11$ in the proof of \cref{th:outerplanar-lb-3}), while the removal of the second vertex from~$\mathcal I$ removes at least $k-1$ internal edges from $G[\mathcal I]$. Since $G$ has $n-3\leq \frac{3k}{2}$ internal edges, and since $2k-1>\frac{3k}{2}$ for $k>2$, it follows that a single vertex is removed from $\mathcal I$. Hence, $|\mathcal I|=n-1\geq  \frac{2k+1}{2k+5}n$, where the last inequality is equivalent to $n\geq \frac{2k+5}{4}$, which holds true by the assumption $n\geq k+2$ and since $k+2\geq \frac{2k+5}{4}$ for $k\geq 0$.

In the inductive case, we have $n\geq \frac{3k}{2}+4$. As in \cref{th:outerplanar-lb-3}, the general strategy consists of defining a $uv$-split subgraph~$H$ of~$G$ with a suitable size and of then selecting a large set $\mathcal I_H\subseteq V(H)$ of vertices, with $u,v \notin \mathcal I_H$, such that $G[\mathcal I_H]$ has degree at most $k$. Applying induction on~$\overline{H}$ to get a set $\mathcal I_{\overline H}$ and then applying \cref{le:tool-degree-k} allows us to obtain the desired set~$\mathcal I$. 

Let $H$ be a $uv$-split subgraph of $G$ with $h$ vertices. A set $\mathcal I_H\subseteq V(H)$ is called \emph{good} for~$H$ if it contains neither $u$ nor $v$, if $G[\mathcal I_H]$ has degree at most $k$, and if:
\begin{itemize}
\item $k+3\leq h\leq \frac{3k}{2}+3$ and $|\mathcal I_H|\geq h-2$; or
\item $\frac{3k}{2}+4\leq h\leq 2k+4$ and $|\mathcal I_H|\geq h-3$; or 
\item $2k+5\leq h\leq 3k+5$ and $|\mathcal I_H|\geq h-4$. 
\end{itemize}
Also, a set $\mathcal I_H\subseteq V(H)$ is an \emph{$(h-2)$-over-$h$ set with $u$} if $|\mathcal I_H|=h-2$, if $u\in \mathcal I_H$ and $v\notin \mathcal I_H$, if $G[\mathcal I_H]$ has degree at most $k$, and if the degree of $u$ in $G[\mathcal I_H]$ is at most $h-k-3$. An \emph{$(h-2)$-over-$h$ set with $v$} is defined analogously,
with $v$ in place of $u$ and vice versa. Finally, a set $\mathcal I_H\subseteq V(H)$ is an \emph{$(h-1)$-over-$h$ set with $u$ and $v$} if $|\mathcal I_H|=h-1$, if $u,v\in \mathcal I_H$, if $G[\mathcal I_H]$ has degree at most $k$, and if the degree of $u$ and $v$ in $G[\mathcal I_H]$ is at most $h-k-2$.

We are going to prove the following three statements.
\begin{itemize}
\item {\bf (S1)} If $h=k+3$, then there exists a good set for $H$.
\item {\bf (S2)} Suppose that:
\begin{itemize}
\item $k+4\leq h\leq \frac{3k}{2}+3$;
\item $G$ has no $st$-split subgraph with $k+3$ vertices; and
\item for every integer $j$ with $k+4\leq j\leq h-1$ and for every $st$-split subgraph $J$ of $G$ with~$j$ vertices, we have that $J$ has a $(j-2)$-over-$j$ set with $s$, a $(j-2)$-over-$j$ set with $t$, and a $(j-1)$-over-$j$ set with $s$ and $t$. 
\end{itemize}
Then there exists a good set for $H$, or there exist an $(h-2)$-over-$h$ set with $u$, an $(h-2)$-over-$h$ set with $v$, and an $(h-1)$-over-$h$ set with $u$ and $v$.
\item {\bf (S3)} Suppose that:
\begin{itemize}
\item $\frac{3k}{2}+4\leq h\leq 3k+5$;
\item $G$ has no $st$-split subgraph with $k+3$ vertices;
\item for every integer $j$ with $k+4\leq j\leq \frac{3k}{2}+3$ and for every $st$-split subgraph $J$ of $G$ with $j$ vertices, we have that $J$ has a $(j-2)$-over-$j$ set with $s$, a $(j-2)$-over-$j$ set with $t$, and a $(j-1)$-over-$j$ set with $s$ and $t$; and
\item for every integer $j$ with $\frac{3k}{2}+4\leq j\leq h-1$, $G$ has no $st$-split subgraph with $j$ vertices. 
\end{itemize}
Then there exists a good set for $H$.
\end{itemize}

Before proving the statements, we show that they imply the theorem (if $k$ is even). 

First, suppose that $G$ has a $uv$-split subgraph $H$ with $k+3$ vertices. By statement (S1), there exists a good set $\mathcal I_H$ for $H$; by definition, $|\mathcal I_H|=k+1$. Hence, $\frac{|\mathcal I_H|}{|V(H)|}=\frac{k+1}{k+3}>\frac{2k+1}{2k+5}$, where the last inequality is true for every value of $k$. Also, $\mathcal I_H$ contains $k+1$ vertices, hence $G[\mathcal I_H]$ has degree at most~$k$. By \cref{le:tool-degree-k}, the set $\mathcal I:=\mathcal I_H\cup \mathcal I_{\overline{H}}$ contains at least $\frac{2k+1}{2k+5}n$ vertices and $G[\mathcal I]$ has degree at most~$k$, as desired. We henceforth assume that $G$ has no edge $uv$ such that a $uv$-split subgraph of $G$ has $k+3$ vertices. 

Second, suppose that $G$ has a $uv$-split subgraph $H$ with $h$ vertices, where $k+4\leq h\leq \frac{3k}{2}+3$, such that there exists no $(h-2)$-over-$h$ set with $u$, or there exists no $(h-2)$-over-$h$ set with~$v$, or there exists no $(h-1)$-over-$h$ set with $u$ and $v$. Among all such subgraphs, choose one such that $h$ is minimum. Then the assumptions of statement (S2) are met, thus there exists a good set $\mathcal I_H$ for~$H$. Hence, $\frac{|\mathcal I_H|}{|V(H)|}=\frac{h-2}{h}>\frac{k+1}{k+3}>\frac{2k+1}{2k+5}$, where the first inequality comes from $h\geq k+4$. Also, $G[\mathcal I_H]$ has degree at most~$k$ since $\mathcal I_H$ is a good set. By \cref{le:tool-degree-k}, the set $\mathcal I:=\mathcal I_H\cup \mathcal I_{\overline{H}}$ contains at least $\frac{2k+1}{2k+5}n$ vertices and $G[\mathcal I]$ has degree at most~$k$, as desired. We henceforth assume that for every integer $h$ with $k+4\leq h\leq \frac{3k}{2}+3$ and for every $uv$-split subgraph $H$ of $G$ with $h$ vertices, $H$ has an $(h-2)$-over-$h$ set with $u$, an $(h-2)$-over-$h$ set with $t$, and an $(h-1)$-over-$h$ set with $u$ and $v$.

Third, setting $\ell=\frac{3k}{2}+2$ and observing that $n\geq \ell+2$ by the hypothesis of the inductive case, \cref{le:decompose-outerplanar} ensures that $G$ contains a $uv$-split subgraph $H$ with a number $h$ of vertices between $\frac{3k}{2}+4$ and~$3k+5$.  Among all such subgraphs, choose one such that $h$ is minimum. Then the assumptions of statement (S3) are met, thus there exists a good set $\mathcal I_H$ for $H$.  If $\frac{3k}{2}+4\leq h\leq 2k+4$, then $\frac{|\mathcal I_H|}{|V(H)|}=\frac{h-3}{h}\geq \frac{(3k/2)+1}{(3k/2)+4}>\frac{2k+1}{2k+5}$, where the last inequality is true for $k\geq 0$, while if $2k+5\leq h\leq 3k+5$, then $\frac{|\mathcal I_H|}{|V(H)|}=\frac{h-4}{h}\geq \frac{2k+1}{2k+5}$. Also, $G[\mathcal I_H]$ has degree at most~$k$ since $\mathcal I_H$ is a good set. By \cref{le:tool-degree-k}, the set $\mathcal I:=\mathcal I_H\cup \mathcal I_{\overline{H}}$ contains at least $\frac{2k+1}{2k+5}n$ vertices and $G[\mathcal I]$ has degree at most~$k$, as desired. This completes the induction.

It remains to prove the three statements. We denote by $w$ the vertex such that the cycle $(u,v,w)$ bounds an internal face of $\mathcal O_H$, by $H_1$ the $uw$-split subgraph of $G$ not containing $v$, and by $H_2$ the $vw$-split subgraph of $G$ not containing $u$. We let $h_1$ and $h_2$ be the number of vertices of $H_1$ and $H_2$, respectively, where $h_1\geq 2$, $h_2\geq 2$,  and $h_1+h_2=h+1$. 

{\bf Statement (S1).} We define $\mathcal I_H:=V(H)\setminus\{u,v\}$. Obviously, $|\mathcal I_H|=k+1=h-2$ and $G[\mathcal I_H]$ has degree at most~$k$, given that $\mathcal I_H$ contains $k+1$ vertices, hence $\mathcal I_H$ is a good set.

{\bf Statement (S2).}  We distinguish three cases. 

\begin{itemize}
\item In the {\em first case}, vertex $w$ has at least $k+1$ neighbors in $V(H)\setminus \{u,v\}$. We prove that the set $\mathcal I^u_H:=V(H)\setminus\{w,v\}$ is an $(h-2)$-over-$h$ set with $u$. The proof exploits the fact that any two vertices of $H$ share at most two neighbors in $H$ (as otherwise such two vertices and their common neighbors would induce $K_{2,3}$, whereas $H$ is an outerplanar graph).

First, $|\mathcal I^u_H|=h-2$, as desired. Second, $u\in \mathcal I^u_H$ and $v\notin \mathcal I^u_H$, by construction. Third, we prove that any vertex $z\neq u$ of $\mathcal I^u_H$ has degree at most $k$ in $G[\mathcal I^u_H]$. Namely, $w$ has at least $k+2$ neighbors in $\mathcal I^u_H$ (at least $k+1$ in $V(H)\setminus \{u,v\}$, by hypothesis, plus $u$). One of such neighbors might be $z$ and at most $2$ of them might be shared with~$z$. Hence, out of the $h-3$ vertices of $\mathcal I^u_H$ different from~$z$, at least $(k-1)$ are different from~$z$ and are not neighbors of~$z$ in $G[\mathcal I^u_H]$, which proves that the degree of $z$ in $G[\mathcal I^u_H]$ is at most $(h-3)-(k-1)$. Since $h\leq \frac{3k}{2}+3$, we have $(h-3)-(k-1)\leq \frac{k}{2}+1\leq k$, where the last inequality is equivalent to $k\geq 2$. Finally, we prove that $u$ has degree at most $h-k-3$ in $G[\mathcal I^u_H]$. Note that $u$ and $w$ share $v$ as a neighbor in $H$, hence out of the at least $k+2$ neighbors of $w$ in $\mathcal I^u_H$, at most one of them is a neighbor of $u$. Hence, out of the $h-3$ vertices of $\mathcal I^u_H$ different from~$u$, at least $k$ are different from $u$ and are not neighbors of $u$ in $G[\mathcal I^u_H]$, which proves that the degree of $u$ in $G[\mathcal I^u_H]$ is at most $(h-3)-k$. Observe that $h-k-3\leq k$ since $h\leq \frac{3k}{2}+3$ and~$k\geq 0$. 

The proof that the set $\mathcal I^v_H:=V(H)\setminus\{u,w\}$ is an $(h-2)$-over-$h$ set with $v$ is symmetric. The proof that the set $\mathcal I^{uv}_H:=V(H)\setminus\{w\}$ is an $(h-1)$-over-$h$ set with $u$ and $v$ is also analogous; in particular, now the degree of any vertex $z\notin \{u,v\}$ of $\mathcal I^{uv}_H$ in $G[\mathcal I^{uv}_H]$ might be as large as $(h-2)-(k-1)=h-k-1$, that is, one more than in $G[\mathcal I^u_H]$. Since $h\leq \frac{3k}{2}+3$, we have $h-k-1\leq \frac{k}{2}+2\leq k$, where the last inequality is equivalent to $k\geq 4$. Also, the degree of $u$ in $\mathcal I^{uv}_H$ is one more than in $G[\mathcal I^u_H]$, since $u$ is a neighbor of $v$, hence it is at most $(h-2)-k$, as required.

\item In the {\em second case}, vertex $w$ has at most $k$ neighbors in $V(H)\setminus \{u,v\}$, $h_1\leq k+2$, and $h_2\leq k+2$. We prove that the set $\mathcal I_H:=V(H)\setminus\{u,v\}$ is a good set for $H$. First, $|\mathcal I_H|=h-2$, as desired. Second, $u,v\notin \mathcal I_H$, by construction. Third, we prove that any vertex $z\in \mathcal I_H$ has degree at most~$k$ in~$G[\mathcal I_H]$. If $z=w$, this comes from the assumption that $w$ has at most~$k$ neighbors in $\mathcal I_H$. If $z\neq w$ and $z\in V(H_1)$ (the case in which $z\neq w$ and $z\in V(H_2)$ is symmetric), then the neighbors of $z$ are all in $V(H_1)$. Since $h_1\leq k+2$ and since $u$ and $z$ are not neighbors of $z$, it follows that the degree of $z$ in $G[\mathcal I_H]$ is at most $k$, as required.

\item In the {\em third case}, vertex $w$ has at most $k$ neighbors in $V(H)\setminus \{u,v\}$, and $h_1\geq k+3$ or $h_2\geq k+3$. By the assumptions of statement (S2), $h_1\neq k+3$ and $h_2\neq k+3$, hence $h_1\geq k+4$ or $h_2\geq k+4$. Since $2k+8>\frac{3}{2}k+3$, for every $k\geq 0$, it follows that $h_1\geq k+4$ and $h_2\geq k+4$ do not both hold. Assume, w.l.o.g., that $h_1\geq k+4$, as the case $h_2\geq k+4$ can be discussed symmetrically. Since $2k+4>\frac{3}{2}k+3$, for every $k\geq 0$, we have $h_2\leq k-1$. 

We prove that $H$ has an $(h-2)$-over-$h$ set with $u$. By the assumptions of statement (S2) and since $k+4\leq h_1\leq h-1$, we have that $H_1$ admits an $(h_1-1)$-over-$h_1$ set with $u$ and $w$, which we denote by $\mathcal I^{uw}_{H_1}$. We show that $\mathcal I^u_H:=\mathcal I^{uw}_{H_1}\cup (V(H_2)\setminus \{v\})$ is an $(h-2)$-over-$h$ set with $u$. First, $|\mathcal I^u_H|=h-2$, given that $\mathcal I^{uw}_{H_1}$ contains all the vertices of $H_1$, except for one, and given that $V(H_2)\setminus \{v\}$ contains all the vertices of $H_2$, except for $v$. Second, $u\in \mathcal I^u_H$, given that $u\in \mathcal I^{uw}_{H_1}$, and $v\notin \mathcal I^u_H$, by construction. Third, we prove that any vertex $z\neq u$ of $\mathcal I^u_H$ has degree at most $k$ in $G[\mathcal I^u_H]$. If $z\in \mathcal I^{uw}_{H_1}$ with $z\neq w$, then all the neighbors of~$z$ are in~$V(H_1)$, hence $z$ has degree at most $k$ in $G[\mathcal I^u_H]$ since it has degree at most~$k$ in~$G[\mathcal I^{uw}_{H_1}]$. Also, if $z\in V(H_2)$ with $z\neq w$, then all the neighbors of $z$ are in $V(H_2)$, hence $z$ has degree at most~$k-3$ in $G[\mathcal I^u_H]$ since $V(H_2)$ contains at most $k-1$ vertices and $v$ is not in $\mathcal I^u_H$. If $z=w$, then $z$ has at most $h_1-k-2$ neighbors in $\mathcal I^{uw}_{H_1}$ and at most $h_2-2$ neighbors in $V(H_2)$ (that is, all the vertices of $V(H_2)$ except for $v$ and $w$ itself). Thus, its degree is at most $h_1-k-2+h_2-2=h-k-3\leq (\frac{3k}{2}+3) -k-3=\frac{k}{2}<k$. Finally, we prove that $u$ has degree at most $h-k-3$ in $G[\mathcal I^u_H]$. We have that all the neighbors of $u$ are in $\mathcal I^{uw}_{H_1}$, given that $w$ and $v$ are the only vertices of $H_2$ that are neighbors of $u$, that $w$ also belongs to $\mathcal I^{uw}_{H_1}$, and that $v$ does not belong to~$\mathcal I^u_H$. It follows that $u$ has degree at most $h_1-k-2\leq h-k-3$ in $G[\mathcal I^u_H]$, as required.

We next prove that $H$ has an $(h-2)$-over-$h$ set with $v$. Let $\mathcal I^{w}_{H_1}$ be an $(h_1-2)$-over-$h_1$ set with $w$ for $H_1$. We show that $\mathcal I^v_H:=\mathcal I^{w}_{H_1}\cup V(H_2)$ is an $(h-2)$-over-$h$ set with $v$. First, $|\mathcal I^v_H|=h-2$, given that $\mathcal I^{w}_{H_1}$ contains all the vertices of $H_1$, except for two. Second, $u\notin \mathcal I^v_H$, since $\mathcal I^{w}_{H_1}$ is an $(h_1-2)$-over-$h_1$ set with $w$, and $v\in \mathcal I^v_H$, by construction. Third, we prove that any vertex $z\neq v$ of $\mathcal I^v_H$ has degree at most $k$ in $G[\mathcal I^v_H]$. If $z\in \mathcal I^w_{H_1}$ with $z\neq w$, then all the neighbors of~$z$ are in~$V(H_1)$, hence $z$ has degree at most $k$ in $G[\mathcal I^v_H]$ since it has degree at most $k$ in~$G[\mathcal I^w_{H_1}]$. Also, if $z\in V(H_2)$ with $z\neq w$, then all the neighbors of $z$ are in $V(H_2)$, hence $z$ has degree at most $k-2$ in $G[\mathcal I^v_H]$ since $V(H_2)$ contains at most $k-1$ vertices. If $z=w$, then $z$ has at most $h_1-k-3$ neighbors in $\mathcal I^w_{H_1}$ and at most $h_2-1$ neighbors in $V(H_2)$ (that is, all the vertices of $V(H_2)$ except for $w$ itself). Thus, its degree is at most $h_1-k-3+h_2-1=h-k-3\leq (\frac{3k}{2}+3) -k-3=\frac{k}{2}<k$. Finally, we prove that $v$ has degree at most $h-k-3$ in $G[\mathcal I^v_H]$. All the neighbors of $v$ are in $V(H_2)$, given that $u$ and $w$ are the only vertices of $H_1$ that are neighbors of~$v$, given that $w$ also belongs to $H_2$, and given that $u$ does not belong to~$\mathcal I^v_H$. It follows that $v$ has degree at most $h_2-1= h-h_1\leq h-k-4$ in $G[\mathcal I^v_H]$, as required.

We now prove that $H$ has an $(h-1)$-over-$h$ set with $u$ and $v$. Let $\mathcal I^{uw}_{H_1}$ be an $(h_1-1)$-over-$h_1$ set with $u$ and $w$ for $H_1$. We show that $\mathcal I^{uv}_H:=\mathcal I^{uw}_{H_1}\cup V(H_2)$ is an $(h-1)$-over-$h$ set with $u$ and $v$. First, $|\mathcal I^{uv}_H|=h-1$, given that $\mathcal I^{uw}_{H_1}$ contains all the vertices of $H_1$, except for one. Second, $u\in \mathcal I^{uv}_H$ and $v\in \mathcal I^{uv}_H$, by construction. Third, we prove that any vertex $z\neq u,v$ of $\mathcal I^{uv}_H$ has degree at most $k$ in $G[\mathcal I^{uv}_H]$. The arguments for the case in which $z\in \mathcal I^{uw}_{H_1}$ (in which $z\in V(H_2)$) with $z\neq w$ are the same as for the proof that $\mathcal I^u_H$ is an $(h-2)$-over-$h$ set with $u$ (that $\mathcal I^v_H$ is an $(h-2)$-over-$h$ set with $v$, respectively). If $z=w$, then $z$ has at most $h_1-k-2$ neighbors in $\mathcal I^{uw}_{H_1}$ and at most $h_2-1$ neighbors in $V(H_2)$ (that is, all the vertices of $V(H_2)$ except for $w$ itself). Thus, its degree is at most $h_1-k-2+h_2-1=h-k-2\leq (\frac{3k}{2}+3) -k-2=\frac{k}{2}+1<k$, where the last inequality holds true since $k\geq 4$. Finally, we prove that $u$ and $v$ have degree at most $h-k-2$ in $G[\mathcal I^{uv}_H]$. We have that all the neighbors of $u$ are in $\mathcal I^{uw}_{H_1}$, except for $v$. It follows that $u$ has degree at most $(h_1-k-2)+1\leq h-k-2$ in $G[\mathcal I^{uv}_H]$, as required. Also, all the neighbors of $v$ are in $V(H_2)$, except for $u$. It follows that $v$ has degree at most $(h_2-1)+1= (h-h_1)+1\leq h-k-3$ in $G[\mathcal I^{uv}_H]$, as required.

{\bf Statement (S3).} We distinguish four cases. 

\begin{itemize}
\item In the {\em first case}, we have $h_1\leq k+2$ and $h_2\leq k+2$. We prove that the set $\mathcal I_H:=V(H)\setminus\{u,w,v\}$ is a good set for $H$. First, $|\mathcal I_H|=h-3$, which is the desired number of vertices, given that $h\geq k+4$, by the assumptions of statement (S3). Second, $u,v\notin \mathcal I_H$, by construction. Third, we prove that any vertex $z\in \mathcal I_H$ has degree at most~$k$ in~$G[\mathcal I_H]$. If $z\in V(H_1)$ (the case in which $z\in V(H_2)$ is symmetric), then the neighbors of $z$ are all in $V(H_1)$. Since $h_1\leq k+2$ and since $u$, $w$, and $z$ are either not in $\mathcal I_H$ or are not neighbors of $z$, it follows that the degree of $z$ in $G[\mathcal I_H]$ is at most $k-1$, as required.
\item In the {\em second case}, we have that $h\leq 2k+4$, $\max\{h_1,h_2\}>k+2$, and $\min\{h_1,h_2\}\leq k+2$. Assume that $h_1>k+2$ and that $h_2\leq k+2$, as the other case is analogous. By the assumptions of statement (S3), we have $k+4\leq h_1 \leq \frac{3k}{2}+3$. Let $\mathcal I^{w}_{H_1}$ be an $(h_1-2)$-over-$h_1$ set with $w$ for $H_1$. We show that $\mathcal I_H:=\mathcal I^{w}_{H_1}\cup (V(H_2)\setminus \{v\})$ is a good set for $H$. First, $|\mathcal I_H|=h-3$, which is equal to the desired number of vertices, given that $h\geq \frac{3k}{2}+4$, by the assumptions of statement (S3). Second, $u\notin \mathcal I_H$, since $\mathcal I^{w}_{H_1}$ is an $(h_1-2)$-over-$h_1$ set with $w$, and $v\notin \mathcal I_H$, by construction. Third, we prove that any vertex $z\in \mathcal I_H$ has degree at most~$k$ in~$G[\mathcal I_H]$. If $z\in \mathcal I^w_{H_1}$ with $z\neq w$, then all the neighbors of~$z$ are in~$V(H_1)$, hence $z$ has degree at most $k$ in $G[\mathcal I_H]$ since it has degree at most $k$ in~$G[\mathcal I^w_{H_1}]$. Also, if $z\in V(H_2)$ with $z\neq w$, then all the neighbors of $z$ are in $V(H_2)$, hence $z$ has degree at most $k$ in $G[\mathcal I_H]$ since $V(H_2)$ contains at most $k+2$ vertices, $v$ is not in $\mathcal I_H$ and $z$ is not a neighbor of itself. Finally, $w$ has at most $h_1-k-3$ neighbors in $\mathcal I^{w}_{H_1}$ and at most $h_2-2$ neighbors in $V(H_2)\setminus \{v\}$, hence its degree in~$G[\mathcal I_H]$ is at most $h_1-k-3+h_2-2=h-k-4\leq k$, where the last inequality uses $h\leq 2k+4$.
\item In the {\em third case}, we have that $h\geq 2k+5$, $\max\{h_1,h_2\}>k+2$, and $\min\{h_1,h_2\}\leq k+2$. Assume that $h_1>k+2$ and that $h_2\leq k+2$, as the other case is analogous. By the assumptions of statement (S3), we have $k+4\leq h_1 \leq \frac{3k}{2}+3$. Let $\mathcal I^{w}_{H_1}$ be an $(h_1-2)$-over-$h_1$ set with $w$ for $H_1$. We show that $\mathcal I_H:=(\mathcal I^{w}_{H_1}\setminus \{w\}) \cup (V(H_2)\setminus \{w,v\})$ is a good set for $H$. First, $|\mathcal I_H|=h-4$, which is equal to the desired number of vertices, given that $h\geq 2k+5$, by the assumptions of statement (S3). Second, $u\notin \mathcal I_H$, since $\mathcal I^{w}_{H_1}$ is an $(h_1-2)$-over-$h_1$ set with $w$, and $v\notin \mathcal I_H$, by construction. Third, we prove that any vertex $z\in \mathcal I_H$ has degree at most~$k$ in~$G[\mathcal I_H]$. If $z\in \mathcal I^w_{H_1}$, then all the neighbors of~$z$ are in~$V(H_1)$, hence $z$ has degree at most $k$ in $G[\mathcal I_H]$ since it has degree at most $k$ in~$G[\mathcal I^w_{H_1}]$. Also, if $z\in V(H_2)$, then all the neighbors of $z$ are in $V(H_2)$, hence $z$ has degree at most $k-1$ in $G[\mathcal I_H]$ since $V(H_2)$ contains at most $k+2$ vertices, $w$ and $v$ are not in $\mathcal I_H$, and $z$ is not a neighbor of itself. 
\item In the {\em fourth case}, we have $\min\{h_1,h_2\}>k+2$. By the assumptions statement (S3), we have $k+4\leq h_1 \leq \frac{3k}{2}+3$ and $k+4\leq h_2 \leq \frac{3k}{2}+3$. Let $\mathcal I^{w}_{H_1}$ be an $(h_1-2)$-over-$h_1$ set with $w$ for $H_1$ and let $\mathcal I^{w}_{H_2}$ be an $(h_2-2)$-over-$h_2$ set with $w$ for $H_2$. We show that $\mathcal I_H:=\mathcal I^{w}_{H_1}\cup \mathcal I^{w}_{H_2}$ is a good set for $H$. First, $|\mathcal I_H|=h-4$, given that $\mathcal I^w_{H_i}$ contains all the vertices of $H_i$, except for two, for $i=1,2$. Second, $u\notin \mathcal I_H$, since $\mathcal I^{w}_{H_1}$ is an $(h_1-2)$-over-$h_1$ set with $w$, and $v\notin \mathcal I_H$, since $\mathcal I^{w}_{H_2}$ is an $(h_2-2)$-over-$h_2$ set with $w$. Third, we prove that any vertex $z\in \mathcal I_H$ has degree at most~$k$ in~$G[\mathcal I_H]$. If $z\in \mathcal I^w_{H_i}$ with $z\neq w$, for some $i\in \{1,2\}$, then all the neighbors of~$z$ are in~$V(H_i)$, hence $z$ has degree at most $k$ in $G[\mathcal I_H]$ since it has degree at most $k$ in~$G[\mathcal I^w_{H_i}]$. Also, $w$ has degree at most $h_1-k-3$ in $G[\mathcal I^{w}_{H_1}]$ and at most $h_2-k-3$  in $G[\mathcal I^{w}_{H_2}]$, hence at most $h_1+h_2-2k-6\leq h-2k-5\leq k$, where we used $h\leq 3k+5$.
\end{itemize}
\end{itemize}
This concludes the proof of the theorem for the case in which $k$ is even. 

If $k$ is odd, then the proof proceeds in the same way with the following differences (as a main numerical difference, the $\frac{3k}{2}$ terms for the even case become $\frac{3k-1}{2}$). First, in the base case we now have $n\leq \frac{3k-1}{2}+3$; this case is handled with the same arguments as when $k$ is even. Second, the definition of a good set $\mathcal I_H$ for $H$ now requires $|\mathcal I_H|\geq h-2$ if $k+3\leq h\leq \frac{3k-1}{2}+3$, requires $|\mathcal I_H|\geq h-3$ if $\frac{3k-1}{2}+4\leq h\leq 2k+4$ and requires $|\mathcal I_H|\geq h-4$ if $2k+5\leq h\leq 3k+4$.  Consequently, in statement (S2), the value of $h$ satisfies $k+4\leq h\leq \frac{3k-1}{2}+3$, while in statement (S3), the value of $h$ satisfies $\frac{3k-1}{2}+4\leq h\leq 3k+4$. Statements (S1)--(S3) then imply the statement of the theorem, by setting $\ell=\frac{3k-1}{2}+2$. In particular, note that $\ell+2=\frac{3k-1}{2}+4$ and $2\ell+1=3k+4$; these are the smallest and largest values of $h$ in statement (S3). The reason why we need to settle for a smaller ratio, namely $\frac{2k+2/3}{2k+14/3}$ rather than $\frac{2k+1}{2k+5}$, when $k$ is odd is that missing  three vertices in a graph with at least $\frac{3k-1}{2}+4$ vertices is not good enough for the ratio $\frac{2k+1}{2k+5}$, while it is for $\frac{2k+2/3}{2k+14/3}$. Indeed, $\frac{(3k-1)/2+1}{(3k-1)/2+4}=\frac{2k+2/3}{2k+14/3}<\frac{2k+1}{2k+5}$. The proof of the statements (S1)--(S3) proceeds almost verbatim with respect to the case in which $k$ is even. In the proof of statement (S2), we used $h\leq \frac{3k}{2}+3$ in order to prove the inequality $h-k-1\leq k$ (and the weaker inequalities $h-k-3\leq k$ and $h-k-2\leq k$). With $h\leq \frac{3k-1}{2}+3$ rather than $h\leq \frac{3k}{2}+3$ such inequality is obviously still satisfied, since the left term is smaller than before. Also, in the proof of statement (S3), we used $h\leq 3k+5$ in order to prove that the degree of $w$, which is upper bounded by $h-2k-5$, is smaller than or equal to $k$. Again, since now $h\leq 3k+5$, the degree of $w$ is still smaller than or equal to $k$.



This concludes the proof of the theorem.
\end{proof}


\section{Induced Subgraphs of Planar Graphs} \label{se:planar}

In this section, we consider the size of the largest induced subgraph of degree at most $k$ that one is guaranteed to find in an $n$-vertex planar graph. \cref{ta:planar} shows the coefficients of the linear terms in the size of the induced subgraphs we can find, for various values of $k$. 
\begin{table}[htb]
\centering
\begin{tabular}{|r||c|c|c|c|c|}
\hline
   & $k=3$ & $k=4$ & $k=5$ & $k=6$ & $k=7$ \\ \hline
LB & $>0.384$     & $>0.457$     & $0.5$     & $>0.571$  & $0.625$     \\ \hline
UB & $<0.572$     & $0.\overline{6}$     & $0.\overline{72}$     & $<0.769$ & $0.8$     \\ \hline
\end{tabular}
\caption{Results on planar graphs with different values of $k$.}
\label{ta:planar}
\vspace{-7mm}
\end{table}

We start by proving two upper bounds, the first one, $\frac{2k-2}{2k+1}n +O(1)$, is tighter for $k<10$, while the second one, $\frac{k+2}{k+4}n +O(1)$, is tighter for $k>10$.

\begin{theorem} \label{th:planar-ub}
For every $n\geq 1$ and for every integer $k\geq 3$, let $\mu=\min \{\frac{2k-2}{2k+1}n +5,\frac{k+2}{k+4}n +2\}$. There exists an $n$-vertex planar graph $G$ with the following property. Consider any set $\mathcal I \subseteq V(G)$ such that $G[\mathcal I]$ has degree at most $k$.  Then $|\mathcal I|\leq \mu$.
\end{theorem}

\begin{proof}
	We distinguish two cases, depending on which of $\frac{2k-2}{2k+1}n +5$ and $\frac{k+2}{k+4}n +2$ is smaller. 

\begin{figure}[htb]
    \centering
	\begin{subfigure}[b]{0.6\textwidth}
		\centering
	\includegraphics[scale=0.7, page=3]{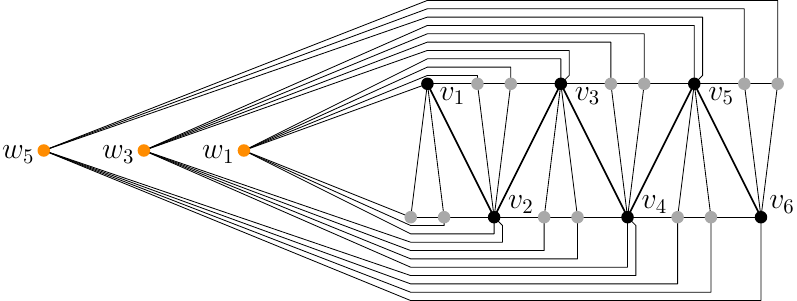}
		\subcaption{}
	\end{subfigure}
	\hfill
	\begin{subfigure}[b]{0.35\textwidth}
		\centering
	\includegraphics[scale=0.7, page=4]{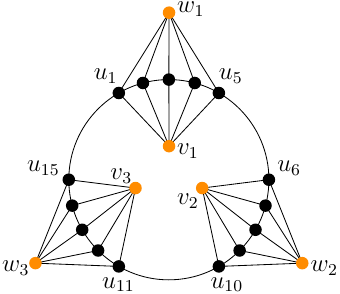}
		\subcaption{}
	\end{subfigure}
	\caption{Illustration for \cref{th:planar-ub}. (a) The graph for the case $\mu=\frac{2k-2}{2k+1}n +5$, with $n=21$, $k=3$, and $h=3$. (b) The graph for the case $\mu=\frac{k+2}{k+4}n +2$, $n=21$, $k=3$, and $h=3$ (values of $k$ for which $\frac{k+2}{k+4}n +2<\frac{2k-2}{2k+1}n +5$ are too large to result in a readable figure).}
    \label{fig:planar-ub-case}
\end{figure}

	Suppose first that $\mu=\frac{2k-2}{2k+1}n +5$. Let $G$ be the $n$-vertex planar graph depicted in~\cref{fig:planar-ub-case}(a) and defined as follows (note that $G$ contains the graph from the proof of \cref{th:outerplanar-ub} as a subgraph). Let $h=\lfloor \frac{n}{2k+1}\rfloor$. Then $G$ contains a path $v_1,\dots,v_{2h}$, vertices $u^i_{1},\dots,u^i_{k-1}$, for $i=1,\dots,2h$, and edges $u^i_{j}u^i_{j+1}$, $v_iu^i_j$, $v_{i-1}u^i_{1}$, and $u^i_{k-1}v_{i+1}$, whenever the end-vertices of such edges are in $G$. The graph also contains vertices $w_1,w_3,\dots,w_{2h-1}$, where $w_{2i-1}$ is connected to $v_{2i-2}$ (if $i>1$), to $v_{2i-1}$, to $v_{2i}$, to $v_{2i+1}$ (if $i<h$), and to $u^{2i-1}_{1},\dots,u^{2i-1}_{k-1}, u^{2i}_{1},\dots,u^{2i}_{k-1}$. Additional $n-h\cdot (2k+1)$ isolated vertices, which do not play any role in the proof of the statement, are added to $G$, so that it has $n$ vertices. 
	
	Consider any set $\mathcal I \subseteq V(G)$ such that $G[\mathcal I]$ has degree at most $k$. We prove that there exists a set $\mathcal I' \subseteq V(G)$ such that $G[\mathcal I']$ has degree at most $k$, such that $|\mathcal I'|=|\mathcal I|$, and such that $\mathcal I'$ does not contain any of $w_1,w_3,\dots,w_{2h-1}$. Suppose that $\mathcal I$ contains a vertex $w_{2i-1}$, for some $i\in \{1,\dots,h\}$. Then there exists a vertex $z$ among $u^{2i-1}_{1},\dots,u^{2i-1}_{k-1},u^{2i}_{1},\dots,u^{2i}_{k-1}$, that is not in $\mathcal I$, as otherwise $w_{2i-1}$ would have degree $2k-2$, which is larger than $k$ since $k\geq 3$. Replace $w_{2i-1}$ in $\mathcal I$ with $z$. Clearly, the size of $\mathcal I$ remains the same. Also, $G[\mathcal I]$ still has degree at most $k$. Indeed, suppose, for a contradiction, that a vertex $z'$ has degree larger than $k$ in $G[\mathcal I]$. Then $z'$ is either $z$ or a neighbor of $z$, given that $z'$ either is not in $\mathcal I$ or is in $\mathcal I$ and has degree at most $k$ before the replacement. If $z'=z$, it suffices to observe that $z'$ has degree at most $4$ in $G$, and one of its neighbors, namely $w_{2i-1}$, is not in $\mathcal I$ after the replacement. If $z'$ is a neighbor of $z$, then $z'$ is a vertex among $v_{2i-2}$ (if $i>1$), $v_{2i-1}$, $v_{2i}$, $v_{2i+1}$ (if $i<h$), and $u^{2i-1}_{1},\dots,u^{2i-1}_{k-1}, u^{2i}_{1},\dots,u^{2i}_{k-1}$. However, all such vertices are neighbors of $w_{2i-1}$, hence the degree of $z'$ in $G[\mathcal I]$ does not increase after the replacement. The repetition of this replacement results in the desired set $\mathcal I'$.

	From the set $\mathcal I'$, a set $\mathcal I'' \subseteq V(G)$ such that $G[\mathcal I'']$ has degree at most $k$, such that $|\mathcal I''|=|\mathcal I'|$, and such that $\mathcal I''$ does not contain any of $w_1,w_3,\dots,w_{2h-1},v_2,\dots,v_{2h-1}$ can be constructed as in the proof of~\cref{th:outerplanar-ub}. Indeed, suppose that $\mathcal I'$ contains a vertex $v_i$, for some $i\in \{2,\dots,2h-1\}$. Then there exists a vertex $z$ among $u^{i-1}_{k-1},u^i_{1},u^i_{2},\dots,u^i_{k-1},u^{i+1}_{1}$ that does not belong to $\mathcal I'$, as otherwise $v_i$ would have degree $k+1$ in $G[\mathcal I']$. Replace $v_i$ in $\mathcal I'$ with $z$. Clearly, the size of $\mathcal I'$ remains the same. Also, $G[\mathcal I']$ still has degree at most $k$. Indeed, no vertex among $w_1,w_3,\dots,w_{2h-1}$ belongs to $\mathcal I'$. Also, every vertex $u^\ell_j$ has degree at most $4$ in $G$ and has one neighbor among $w_1,w_3,\dots,w_{2h-1}$, which is not in $\mathcal I'$, hence the degree of $u^\ell_j$ in $G[\mathcal I']$ is at most $3$. Finally, $z$ is a neighbor of $v_i$, which however is not in $\mathcal I'$ after the replacement, and might be a neighbor of $v_{i-1}$ or $v_{i+1}$. However, $v_{i-1}$ or $v_{i+1}$ are also neighbors of $v_i$, hence if they are in $\mathcal I'$, their degree in $G[\mathcal I']$ does not increase. The repetition of this replacement results in the desired set $\mathcal I''$. 
	
	Note that $|\mathcal I''|\leq n-(2h-2)-h=n-3 \lfloor \frac{n}{2k+1}\rfloor+2 \leq n-\frac{3n}{2k+1}+5=\frac{2k-2}{2k+1}n+5$. 
	
	Suppose next that $\mu=\frac{k+2}{k+4}n +2$. We define a planar graph $G$ as follows, see \cref{fig:planar-ub-case}(b). Let $h=\lfloor \frac{n}{k+4}\rfloor$. Initialize $G$ to a cycle with $n-2h$ vertices, labeled $u_1,\dots,u_{n-2h}$. For $i=1,\dots,h$, insert into $G$ two vertices $v_i$ and $w_i$, which are adjacent to the $k+2$ vertices $u_j$ such $(i-1)\cdot (k+2)+1\leq j \leq i \cdot (k+2)$. 
	
	Consider any set $\mathcal I \subseteq V(G)$ such that $G[\mathcal I]$ has degree at most $k$. For $i=1,\dots,h$, we have that, among the $k+4$ vertices $v_i,w_i,u_{(i-1)\cdot (k+2)+1},\dots,u_{i \cdot (k+2)}$, at least two do not belong to $\mathcal I$. Indeed, if at least $k+1$ among $u_{(i-1)\cdot (k+2)+1},\dots,u_{i \cdot (k+2)}$ belong to $\mathcal I$, then neither $v_i$ nor $w_i$ belongs to $\mathcal I$, as otherwise $G[\mathcal I]$ would contain a vertex of degree at least $k+1$. It follows that $|\mathcal I|\leq n-2h =n-2\lfloor \frac{n}{k+4}\rfloor\leq n-2\frac{n}{k+4}+2=\frac{k+2}{k+4}n+2$.
\end{proof}


\newcommand{\pgr}{\textsc{GreedyRemoval}\xspace}

We now focus our attention on the lower bound. We present an algorithm, called \pgr, that given a (not necessarily planar) graph $G$ and a target integer value $k\geq 3$, defines a large set of vertices $\mathcal I\subseteq V(G)$ such that $G[\mathcal I]$ has degree at most $k$. As the name suggests, the algorithm follows a natural greedy strategy: It repeatedly removes from the current graph $\mathcal G$ (initially, we have $\mathcal G=G$) any vertex with maximum degree, until $\mathcal G$ has the target degree $k$. When this happens, the set of vertices of $\mathcal G$ is the desired set~$\mathcal I$. 

In order to get better bounds, we tweak this strategy a bit if $k=3$ or $k=4$. In these cases, once~$\mathcal G$  has degree $k+1$, we either repeatedly remove vertices of degree $k+1$, or we act as follows. We 
augment $\mathcal G$ to a $(k+1)$-regular graph, in which we find a small dominating set $\mathcal D$. We remove from $\mathcal G$ the vertices in $\mathcal D$. Now $\mathcal G$ has degree $k$ and its vertex set is the desired set~$\mathcal I$. We choose the strategy that removes the smallest number of vertices. This algorithm is the basis for the following theorem.

\begin{theorem} \label{th:planar-lb}
	For every $n\geq 1$, for every graph $G$ with $n$ vertices, $m$ edges, and degree $d \geq 4$, and for every integer $3 \leq k \leq d-1$, there exists a set $\mathcal I \subseteq V(G)$ such that $G[\mathcal I]$ has degree at most $k$ and such that:
	\begin{enumerate}
		\item\label{th:planar-lb-d4-k3} 
		if $d=4$ and $k=3$, then $|\mathcal I|\geq \frac{7}{11}n-\frac{24}{11}$;
		\item\label{th:planar-lb-d5-k3}
		if $d=5$ and $k=3$, then $|\mathcal I|\geq \frac{35}{78}n-\frac{8}{13}$;
		\item\label{th:planar-lb-d6-k3}
		if $d\geq 6$ and $k=3$, then $|\mathcal I|\geq  \frac{10}{13}n - \frac{5}{39}m -\frac{8}{13}$;
		\item\label{th:planar-lb-d5-k4} 
		if $d=5$ and $k=4$, then $|\mathcal I|\geq \frac{9}{14}n-\frac{1}{2}$;
		\item\label{th:planar-lb-d6-k4} 
		if $d\geq 6$ and $k=4$, then $|\mathcal I|\geq \frac{54}{59}n - \frac{9}{59}m - \frac{35}{59}$; and
		\item\label{th:planar-lb-d6-k5} 
		if $d\geq 6$ and $k\geq 5$, then $|\mathcal I|\geq n-\frac{m}{k+1}$.		
	\end{enumerate}
	
\end{theorem}

\begin{proof}
Before delving into the analysis of \pgr, we introduce some notation. A simple, yet crucial, observation is that the degree of the current graph $\mathcal G$ does not increase throughout the algorithm execution. For any $i=d,d-1,\dots,k$, we let $\mathcal G_i$ be the graph $\mathcal G$ the first time it has degree at most~$i$; thus, $\mathcal G_d=G$ and $\mathcal G_k$ is the graph whose vertex set is $\I$. It is possible that $\mathcal G_i=\mathcal G_{i-1}$; this happens if removing vertices of degree~$j$ from $\mathcal G_j$, for some $j\geq i+1$, the degree of $\mathcal G$ drops to $i-1$, or lower. We denote by $r_i$ the number of vertices of degree $i$ that are removed by \pgr; these are the vertices that are in $\mathcal G_i$ and not in $\mathcal G_{i-1}$. Also, we denote by $r_6^+$ the sum $\sum_{j=6}^d r_j$, that is, the number of vertices of degree at least $6$ that are removed by \pgr. 


We are now ready to analyze \pgr. In order to do so, we distinguish the six cases in the theorem's statement.

	
\textbf{Case 1:} $d=4$ and $k=3$. We show that $G=\mathcal G_4$ can be augmented to a $4$-regular graph by adding at most $6$ vertices, and several edges, to it. While there exist two vertices that are not adjacent and that both have degree smaller than $4$, we add an edge between them. When no such an edge can be added, the vertices of degree smaller than $4$ induce a clique $K_x$ in $\mathcal G_4$, with $x\leq 4$. Then it suffices to add at most $6$ vertices (and some further edges) to $\mathcal G_4$ in order to obtain a $4$-regular graph $\mathcal H_4$, see \cref{fig:planar-lb-4-3}. 
\begin{figure}[htb]
	\centering
	\begin{subfigure}[b]{0.1\textwidth}
		\centering
		\includegraphics[page=1]{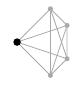}
		\subcaption{}
	\end{subfigure}
	\hfill
	\begin{subfigure}[b]{0.1\textwidth}
		\centering
		\includegraphics[page=2]{img/lb_planar.pdf}
		\subcaption{}
	\end{subfigure}
	\hfill
	\begin{subfigure}[b]{0.1\textwidth}
		\centering
		\includegraphics[page=3]{img/lb_planar.pdf}
		\subcaption{}
	\end{subfigure}
	\hfill
	\begin{subfigure}[b]{0.1\textwidth}
		\centering
		\includegraphics[page=4]{img/lb_planar.pdf}
		\subcaption{}
	\end{subfigure}
	\hfill
	\begin{subfigure}[b]{0.1\textwidth}
		\centering
		\includegraphics[page=5]{img/lb_planar.pdf}
		\subcaption{}
	\end{subfigure}
	\hfill
	\begin{subfigure}[b]{0.1\textwidth}
		\centering
		\includegraphics[page=6]{img/lb_planar.pdf}
		\subcaption{}
	\end{subfigure}
	\hfill
	\begin{subfigure}[b]{0.1\textwidth}
		\centering
		\includegraphics[page=7]{img/lb_planar.pdf}
		\subcaption{}
	\end{subfigure}
	\hfill
	\begin{subfigure}[b]{0.1\textwidth}
		\centering
		\includegraphics[page=8]{img/lb_planar.pdf}
		\subcaption{}
	\end{subfigure}
	\hfill
	\begin{subfigure}[b]{0.1\textwidth}
		\centering
		\includegraphics[page=9]{img/lb_planar.pdf}
		\subcaption{}
	\end{subfigure}
	\caption{If the vertices of degree smaller than $4$ induce a clique $K_x$ in $\mathcal G_4$, it suffices to add at most $6$ vertices (and some edges) to $\mathcal G_4$ in order to obtain a $4$-regular graph $\mathcal H_4$. Vertices and edges in $\mathcal G_4$ are black, while vertices and edges in $\mathcal H_4$ and not in $\mathcal G_4$ are gray. Each vertex of $K_x$ has degree at least $x-1$ and at most $3$ in $\mathcal G_4$. Also, the sum of the degrees of the vertices in $K_x$ is even. This results in the following cases. (a) $x=1$, one vertex of degree $0$. (b) $x=1$, one vertex of degree $2$. (c) $x=2$, two vertices of degree $1$. (d) $x=2$, two vertices of degree $2$. (e) $x=2$, one vertex of degree $1$ and one vertex of degree $3$. (f) $x=2$, two vertices of degree $3$. (g) $x=3$, three vertices of degree $2$. (h) $x=3$, two vertices of degree $3$ and one vertex of degree $2$. (i) $x=4$, four vertices of degree $3$.}
	\label{fig:planar-lb-4-3}
\end{figure}
Every $h$-vertex $4$-regular graph has a dominating set of size at most $\frac{4}{11}h$~\cite{liu2004domination}, hence $\mathcal H_4$ has a dominating set $\mathcal D$ of size at most $\frac{4}{11}(n+6)$. Let $r_4$ be the number of vertices of $\mathcal G_4$ that are in $\mathcal D$. Removing such vertices from $\mathcal G_4$ results in a graph $\mathcal G_3$ with degree at most $3$. We thus get the following:
\begin{equation}\label{eq:case1}
	r_4 \leq \frac{4}{11}(n + 6).
\end{equation}
The vertex set of $\mathcal G_3$ is our desired set $\I$, where $|\I| = n-r_4\geq \frac{7}{11}n-\frac{24}{11}$.

\textbf{Case 2:} $d=5$ and $k=3$. The algorithm first removes $r_5$ vertices from $G$ obtaining $\mathcal G_4$, and then removes $r_4$ vertices from $\mathcal G_4$ resulting in $\mathcal G_3$, whose vertex set is $\I$. We upper bound $r_4 + r_5$ in two different ways. First, by \cref{eq:case1} and since the number of vertices of $\mathcal G_4$ is $n-r_5$, we get: 

\begin{equation}\label{eq:case2-1}
r_4 + r_5 \leq \frac{4}{11}(n - r_5 + 6) + r_5 = \frac{4n + 7r_5+24}{11}.
\end{equation}
Second, the number of vertices of degree $4$ that are removed from $\mathcal G_4$ is at most $\frac{1}{4}$ times the number of edges of $\mathcal G_4$, since the removal of a degree-$4$ vertex removes $4$ edges from $\mathcal G_4$. Also, the number of edges of $\mathcal G_4$ is equal to the number of edges of $\mathcal G_5$, which is at most $\frac{5n}{2}$ given that $\mathcal G_5$ has degree $5$, minus $5$ times the number of vertices of degree $5$ that are removed from $\mathcal G_5$. Hence, we get:
\begin{equation}\label{eq:case2-2}
	r_4 + r_5 \leq \frac{\frac{5n}{2} - 5r_5}{4} + r_5 = \frac{5n - 2r_5}{8}.
\end{equation}
We thus get:
\begin{equation}\label{eq:case2}
	r_4 + r_5 \leq \frac{43}{78}n+\frac{8}{13}.
\end{equation} 
\cref{eq:case2} comes from \cref{eq:case2-1} when $r_5 \leq \frac{23n}{78}-\frac{32}{13}$ and from \cref{eq:case2-2} when $r_5 \geq \frac{23n}{78}-\frac{32}{13}$. Therefore, it follows that $|\mathcal I|=n-r_4-r_5\geq \frac{35}{78}n-\frac{8}{13}$.

	
\textbf{Case 3:} $d\geq 6$ and $k=3$. The algorithm first removes $r_{6^+}$ vertices from $G$ to obtain~$\mathcal G_5$, then removes $r_5$ vertices from $\mathcal G_5$ to obtain $\mathcal G_4$, and finally removes $r_4$ vertices from $\mathcal G_4$, resulting in $\mathcal G_3$, whose vertex set is $\I$. We upper bound $r_4 + r_5 + r_{6^+}$ in three different ways. First, by \cref{eq:case2} and since the number of vertices of~$\mathcal G_5$ is $n-r_{6^+}$ we get: 
\begin{equation}\label{eq:case3-1}
	r_4 + r_5 + r_{6^+} \leq \frac{43}{78} (n - r_{6^+}) + \frac{8}{13} + r_{6^+} = \frac{43 n + 35 r_{6^+}+48}{78}.
\end{equation} 
Second, by \cref{eq:case1} and since the number of vertices of~$\mathcal G_4$ is $n-r_5-r_{6^+}$, we get: 
\begin{equation}\label{eq:case3-2}
	r_4 + r_5 + r_{6^+} \leq \frac{4}{11} (n - r_5 - r_{6^+}+6) + r_5 + r_{6^+} = \frac{4 n + 7 r_5 + 7 r_{6^+}+24}{11}.
\end{equation}
Third, as in Case 2, the number of vertices of degree $4$ that are removed from $\mathcal G_4$ is at most $\frac{1}{4}$ times the number of edges of $\mathcal G_4$. Also, the number of edges of $\mathcal G_4$ is at most the number $m$ of edges of $G$, minus $5$ times the number of vertices of degree $5$ that are removed from $\mathcal G_5$, minus $6$ times the number of vertices of degree at least $6$ that are removed from $G$. Hence, we get:
\begin{equation}\label{eq:case3-3}
	r_4 + r_5 + r_{6^+} \leq  \frac{m - 5r_5 - 6r_{6^+}}{4} + r_5 + r_{6^+} = \frac{m - r_5 - 2 r_{6^+}}{4}.
\end{equation}
We thus get:
\begin{equation}\label{eq:case3}
	r_4 + r_5 + r_{6^+} \leq \frac{3}{13}n + \frac{5}{39}m +\frac{8}{13}.
\end{equation} 
\cref{eq:case3} comes from \cref{eq:case3-1} when $r_{6^+} \leq -\frac{5}{7}n + \frac{2}{7}m$,\ from \cref{eq:case3-2} when $r_5+r_{6^+} \leq -\frac{19}{91}n+\frac{55}{273}m-\frac{32}{13}$, and it comes from \cref{eq:case3-3} when $r_{6^+} \geq -\frac{5}{7}n + \frac{2}{7}m$ and $r_5+r_{6^+} \geq -\frac{19}{91}n+\frac{55}{273}m-\frac{32}{13}$. The last implication is justified by observing that, for any fixed value of $r_5+r_{6^+}$, the quantity $\frac{m - r_5 - 2 r_{6^+}}{4}$ is maximized by choosing $r_{6^+}$ as small as possible, i.e., equal to $-\frac{5}{7}n + \frac{2}{7}m$. Then the value of $r_5$ which maximizes $\frac{m - r_5 - 2 r_{6^+}}{4}$ is $-\frac{19}{91}n+\frac{55}{273}m-\frac{32}{13} - (-\frac{5}{7}n + \frac{2}{7}m) =\frac{46}{91}n-\frac{23}{273}m - \frac{32}{13}$. It follows that $|\mathcal I|=n-r_4-r_5-r_{6^+}\geq \frac{10}{13}n - \frac{5}{39}m -\frac{8}{13}$. 	


	
\begin{figure}[htb]
	\centering
	\begin{subfigure}[b]{0.115\textwidth}
		\centering
		\includegraphics[scale=.7, page=11]{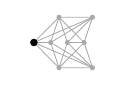}
		\subcaption{}
	\end{subfigure}
	\hfill
	\begin{subfigure}[b]{0.115\textwidth}
		\centering
		\includegraphics[scale=.7, page=12]{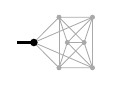}
		\subcaption{}
	\end{subfigure}
	\hfill
	\begin{subfigure}[b]{0.115\textwidth}
		\centering
		\includegraphics[scale=.7, page=13]{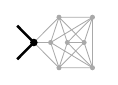}
		\subcaption{}
	\end{subfigure}
	\hfill
	\begin{subfigure}[b]{0.115\textwidth}
		\centering
		\includegraphics[scale=.7, page=14]{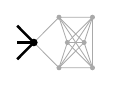}
		\subcaption{}
	\end{subfigure}
	\hfill
	\begin{subfigure}[b]{0.115\textwidth}
		\centering
		\includegraphics[scale=.7, page=15]{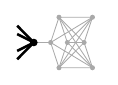}
		\subcaption{}
	\end{subfigure}
	\hfill
	\begin{subfigure}[b]{0.115\textwidth}
		\centering
		\includegraphics[scale=.7, page=16]{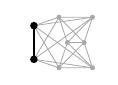}
		\subcaption{}
	\end{subfigure}
	\hfill
	\begin{subfigure}[b]{0.115\textwidth}
		\centering
		\includegraphics[scale=.7, page=17]{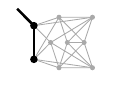}
		\subcaption{}
	\end{subfigure}
    \hfill
	\begin{subfigure}[b]{0.115\textwidth}
		\centering
		\includegraphics[scale=.7, page=18]{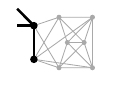}
		\subcaption{}
	\end{subfigure}
\bigskip
	\begin{subfigure}[b]{0.115\textwidth}
		\centering
		\includegraphics[scale=.7, page=19]{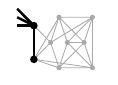}
		\subcaption{}
	\end{subfigure}
	\hfill
	\begin{subfigure}[b]{0.115\textwidth}
		\centering
		\includegraphics[scale=.7, page=20]{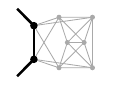}
		\subcaption{}
	\end{subfigure}
	\hfill
	\begin{subfigure}[b]{0.115\textwidth}
		\centering
		\includegraphics[scale=.7, page=21]{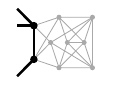}
		\subcaption{}
	\end{subfigure}
	\hfill
	\begin{subfigure}[b]{0.115\textwidth}
		\centering
		\includegraphics[scale=.7, page=22]{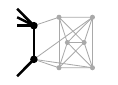}
		\subcaption{}
	\end{subfigure}
	\hfill
	\begin{subfigure}[b]{0.115\textwidth}
		\centering
		\includegraphics[scale=.7, page=23]{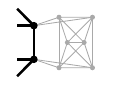}
		\subcaption{}
	\end{subfigure}
	\hfill
	\begin{subfigure}[b]{0.115\textwidth}
		\centering
		\includegraphics[scale=.7, page=24]{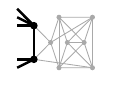}
		\subcaption{}
	\end{subfigure}
    \hfill
    \begin{subfigure}[b]{0.115\textwidth}
		\centering
		\includegraphics[scale=.7, page=25]{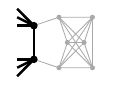}
		\subcaption{}
	\end{subfigure}
	\hfill
	\begin{subfigure}[b]{0.115\textwidth}
		\centering
		\includegraphics[scale=.7, page=26]{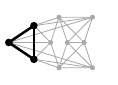}
		\subcaption{}
	\end{subfigure}
	\bigskip
    \begin{subfigure}[b]{0.115\textwidth}
		\centering
		\includegraphics[scale=.7, page=27]{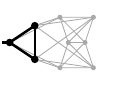}
		\subcaption{}
	\end{subfigure}
	\hfill
	\begin{subfigure}[b]{0.115\textwidth}
		\centering
		\includegraphics[scale=.7, page=32]{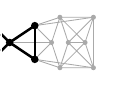}
		\subcaption{}
	\end{subfigure}
	\hfill
	\begin{subfigure}[b]{0.115\textwidth}
		\centering
		\includegraphics[scale=.7, page=28]{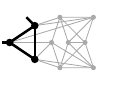}
		\subcaption{}
	\end{subfigure}
	\hfill
	\begin{subfigure}[b]{0.115\textwidth}
		\centering
		\includegraphics[scale=.7, page=33]{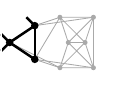}
		\subcaption{}
	\end{subfigure}
	\hfill
	\begin{subfigure}[b]{0.115\textwidth}
		\centering
		\includegraphics[scale=.7, page=35]{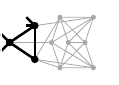}
		\subcaption{}
	\end{subfigure}
    \hfill
    \begin{subfigure}[b]{0.115\textwidth}
		\centering
		\includegraphics[scale=.7, page=29]{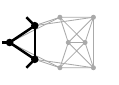}
		\subcaption{}
	\end{subfigure}
	\hfill
	\begin{subfigure}[b]{0.115\textwidth}
		\centering
		\includegraphics[scale=.7, page=34]{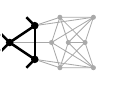}
		\subcaption{}
	\end{subfigure}
	\hfill
	\begin{subfigure}[b]{0.115\textwidth}
		\centering
		\includegraphics[scale=.7, page=36]{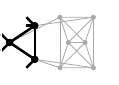}
		\subcaption{}
	\end{subfigure}
\bigskip
    \begin{subfigure}[b]{0.115\textwidth}
		\centering
		\includegraphics[scale=.7, page=37]{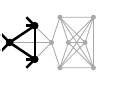}
		\subcaption{}
	\end{subfigure}
	\hfill
	\begin{subfigure}[b]{0.115\textwidth}
		\centering
		\includegraphics[scale=.7, page=38]{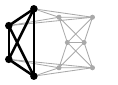}
		\subcaption{}
	\end{subfigure}
	\begin{subfigure}[b]{0.115\textwidth}
		\centering
		\includegraphics[scale=.7, page=39]{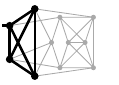}
		\subcaption{}
	\end{subfigure}
	\hfill
	\begin{subfigure}[b]{0.115\textwidth}
		\centering
		\includegraphics[scale=.7, page=40]{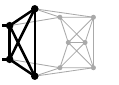}
		\subcaption{}
	\end{subfigure}
	\hfill
	\begin{subfigure}[b]{0.115\textwidth}
		\centering
		\includegraphics[scale=.7, page=41]{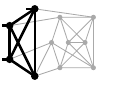}
		\subcaption{}
	\end{subfigure}
	\hfill
	\begin{subfigure}[b]{0.115\textwidth}
		\centering
		\includegraphics[scale=.7, page=42]{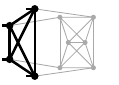}
		\subcaption{}
	\end{subfigure}
    \hfill
    \begin{subfigure}[b]{0.115\textwidth}
		\centering
		\includegraphics[scale=.7, page=43]{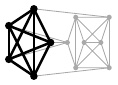}
		\subcaption{}
	\end{subfigure}
	\caption{If the vertices of degree smaller than $5$ induce a clique $K_x$ in $\mathcal G_5$, it suffices to add at most $7$ vertices (and some edges) to $\mathcal G_5$ in order to obtain a $5$-regular graph $\mathcal H_5$. Vertices and edges in $\mathcal G_5$ are black, while vertices and edges in $\mathcal H_5$ and not in $\mathcal G_5$ are gray. Each vertex of $K_x$ has degree at least $x-1$ and at most $4$ in $\mathcal G_5$. This results in the following cases. 
 (a) $x=1$, one deg-$0$ vertex. 
 (b) $x=1$, one deg-$1$ vertex.
 (c) $x=1$, one deg-$2$ vertex. 
 (d) $x=1$, one deg-$3$ vertex. 
 (e) $x=1$, one deg-$4$ vertex. 
 (f) $x=2$, two deg-$1$ vertices. 
 (g) $x=2$, one deg-$2$ vertex and one deg-$1$ vertex. 
 (h) $x=2$, one deg-$3$ vertex and one deg-$1$ vertex.
 (i) $x=2$, one deg-$4$ vertex and one deg-$1$ vertex.
 (j) $x=2$, two deg-$2$ vertices.
 (k) $x=2$, one deg-$3$ vertex and one deg-$2$ vertex.
 (l) $x=2$, one deg-$4$ vertex and one deg-$2$ vertex.
 (m) $x=2$, two deg-$3$ vertices.
 (n) $x=2$, one deg-$4$ vertex and one deg-$3$ vertex.
 (o) $x=2$, two deg-$4$ vertices.
 (p) $x=3$, three deg-$2$ vertices.
 (q) $x=3$, one deg-$3$ vertex and two deg-$2$ vertices.
 (r) $x=3$, one deg-$4$ vertex and two deg-$2$ vertices.
 (s) $x=3$, two deg-$3$ vertices and one deg-$2$ vertex.
 (t) $x=3$, one deg-$4$ vertex, one deg-$3$ vertex, and one vertex of degree $2$.
 (u) $x=3$, two vertices of degree $4$ and one deg-$2$ vertex.
 (v) $x=3$, three deg-$3$ vertices.
 (w) $x=3$, one deg-$4$ vertex and two deg-$3$ vertices.
 (x) $x=3$, two deg-$4$ vertices and one deg-$3$ vertex.
 (y) $x=3$, three deg-$4$ vertices.
 (z) $x=4$, four deg-$3$ vertices.
 (aa) $x=4$, one deg-$4$ vertex and three deg-$3$ vertices.
 (ab) $x=4$, two deg-$4$ vertices and one deg-$3$ vertex.
 (ac) $x=4$, three deg-$4$ vertices and three deg-$3$ vertices.
 (ad) $x=4$, four deg-$4$ vertices.
 (ae) $x=5$, five deg-$4$ vertices.
 }
		\label{fig:planar-lb-5-4}
\end{figure}

	\textbf{Case 4:} $d=5$ and $k=4$. This case is similar to \textbf{Case 1}. Namely, we augment $G=\mathcal G_5$ to a $5$-regular graph by adding at most $7$ vertices, and several edges, to it. This is done by first adding edges between pairs of non-adjacent vertices whose degree is smaller than $5$. When this is not possible, the vertices of degree smaller than $5$ induce a clique $K_x$ in $\mathcal G_5$. Then it suffices to add at most $7$ vertices (and some further edges) to $\mathcal G_5$ in order to obtain a $5$-regular graph $\mathcal H_5$, see \cref{fig:planar-lb-5-4}. Every $h$-vertex $5$-regular graph has a dominating set of size at most $\frac{5}{14}h$~\cite{xing2006upper}, hence $\mathcal H_5$ has a dominating set $\mathcal D$ of size at most $\frac{5}{14}(n+7)$. Let $r_5$ be the number of vertices of $\mathcal G_5$ that are in $\mathcal D$. Removing such vertices from $\mathcal G_5$ results in a graph $\mathcal G_4$ with degree at most $4$. We thus get the following:
	\begin{equation}\label{eq:case4}
		r_5 \leq \frac{5}{14}(n + 7).
	\end{equation}
	The vertex set of $\mathcal G_4$ is our desired set $\I$, where $|\I| = n-r_5\geq \frac{9}{14}n-\frac{1}{2}$.


\textbf{Case 5:} $d\geq 6$ and $k=4$. The algorithm first removes $r_{6^+}$ vertices from $G$ obtaining $\mathcal G_5$, and then removes $r_5$ vertices from $\mathcal G_5$ resulting in $\mathcal G_4$, whose vertex set is $\I$. We upper bound $r_5 + r_{6^+}$ in two different ways. First, by \cref{eq:case4} and since the number of vertices of $\mathcal G_5$ is $n-r_{6^+}$, we get: 
\begin{equation}\label{eq:case5-1}
	r_5 + r_{6^+} \leq \frac{5}{14}(n - r_{6^+} + 7) + r_{6^+} = \frac{5n + 9r_{6^+}+35}{14}.
\end{equation}
Second, the number of vertices of degree $5$ that are removed from $\mathcal G_5$ is at most $\frac{1}{5}$ times the number of edges of $\mathcal G_5$, since the removal of a degree-$5$ vertex removes $5$ edges from $\mathcal G_5$. Also, the number of edges of $\mathcal G_5$ is at most the number $m$ of edges of $G$, minus $6$ times the number of vertices of degree at least $6$ that are removed from $G$. Hence, we get:
\begin{equation}\label{eq:case5-2}
	r_5 + r_{6^+} \leq \frac{m - 6r_{6^+}}{5} + r_{6^+} = \frac{m - r_{6^+}}{5}.
\end{equation}
We thus get:
\begin{equation}\label{eq:case5}
	r_5 + r_{6^+} \leq \frac{5}{59}n +\frac{9}{59}m + \frac{35}{59}.
\end{equation} 
\cref{eq:case5} comes from \cref{eq:case5-1} when $r_{6^+} \leq -\frac{25}{59}n + \frac{14}{59}m -\frac{175}{59}$ and from \cref{eq:case5-2} when $r_{6^+} \geq -\frac{25}{59}n + \frac{14}{59}m -\frac{175}{59}$. Therefore, it follows that $|\mathcal I|=n-r_5 - r_{6^+}\geq \frac{54}{59}n - \frac{9}{59}m - \frac{35}{59}$.

	
\textbf{Case 6:} $d\geq 6$ and $k\geq 5$. The number of vertices of degree at least $k+1$ that are removed from $G$ is at most $\frac{1}{k+1}$ times the number of edges of $G$, since the removal of a vertex with degree at least $k+1$ removes at least $k+1$ edges from $G$. Hence, we get:
\begin{equation}\label{eq:case6}
	\sum_{j=k+1}^d r_j  \leq \frac{m}{k+1}.
\end{equation}
Therefore, it follows that $|\mathcal I|=n-\sum_{j=k+1}^d r_j\geq n-\frac{m}{k+1}$.
\end{proof}

Since a planar graph has at most $3n-6$ edges, we get the following corollary.

\begin{corollary} \label{cor:planar-greedy}
For every~$n\geq 1$, for every planar graph~$G$ with~$n$ vertices, and for every integer~$k\geq 3$, there exists a set~$\mathcal I \subseteq V(G)$ such that~$G[\mathcal I]$ has degree at most~$k$ and such that (i) if~$k=3$, then~$|\mathcal I|\geq  \frac{5}{13}n+\frac{2}{13}$; (ii) if~$k=4$, then~$|\mathcal I|\geq\frac{27}{59}n+\frac{19}{59}$; and (iii) if~$k\geq 5$, then~$|\mathcal I|\geq \frac{k-2}{k+1}n+\frac{6}{k+1}$.	
\end{corollary}

Similarly, we can get better lower bounds for sparse families of planar graphs. For example, $n$-vertex planar graphs that admit a planar drawing in which every face has degree at least $g$ (these include planar graphs of girth at least $g$) have at most $\frac{g(n-2)}{g-2}$ edges by Euler's formula, and hence have a set $\mathcal I \subseteq V(G)$ such that $G[\mathcal I]$ has degree at most $k$ and such that $|\mathcal I|\geq \frac{(25g-60)n-14g+48}{39(g-2)}$, $|\mathcal I|\geq \frac{(45g-108)n -17g +70}{59(g-2)}$, and $|\mathcal I|\geq \frac{(gk -2k -2)n +2g}{(g-2)(k+1)}$ if $k=3$, $k=4$, or $k\geq 5$, respectively, by \cref{th:planar-lb}. As a notable implication, $n$-vertex bipartite planar graphs have a set $\mathcal I \subseteq V(G)$ such that $G[\mathcal I]$ has degree at most $3$ and such that $|\mathcal I|\geq \frac{20}{39}n-\frac{4}{39}$, which is strictly larger than the maximum size of an induced linear forest~\cite{DBLP:journals/dm/DrossMP19}.


\section{Conclusions}\label{se:conclusions}
In this paper, we initiated the study of large induced graphs of bounded degree in classes of planar graphs. For any integer $k\geq 3$, we presented upper and lower bounds for the value of the function $f_k(n)$ such that every $n$-vertex outerplanar or planar graph has an induced subgraph with degree at most $k$ and with $f_k(n)$ vertices. Our upper bounds are existential, while our lower bounds are universal and are based on polynomial-time algorithms that compute the desired graphs. An obvious direction for further research is to close the gaps between the upper and lower bounds we presented. See again \cref{ta:degree-3} in \cref{se:introduction} for a numerical comparison of our upper and lower bounds for the case $k=3$.

We believe that our techniques have the potential to yield further improvements on the considered problems, as explained below.

Let $c_k$ be the maximum value such that every $n$-vertex outerplanar graph $G$ has an induced subgraph with degree at most $k$ and with $c_k\cdot n +o(n)$ vertices. We believe that the use of \cref{le:decompose-outerplanar,le:tool-degree-k} that was done in order to prove \cref{th:outerplanar-lb-3,th:outerplanar-lb-k} could be stressed further, in order to prove a lower bound of the form $d_k\cdot n + O(1)$ for the size of a set $\mathcal I$ such that $G[\mathcal I]$ has degree at most $k$, where $d_k$ is arbitrarily close to $c_k$. Indeed, by employing ``very large'' values of $\ell$ in \cref{le:decompose-outerplanar}, the numerical effect of the constraint imposed by \cref{le:tool-degree-k} of excluding two vertices, namely $u$ and $v$, from the set of vertices $\mathcal I_H$ that induce a large subgraph $G[\mathcal I_H]$  with degree at most $k$ in the $uv$-split subgraph $H$ of $G$ becomes negligible. We plan to use this intuition to experimentally determine the value $c_k$, for small values of $k$, up to some decimal digits. In particular, it would be interesting to understand whether $c_k$ is equal to $\frac{k-1}{k}$, which would match our upper bound from \cref{th:outerplanar-ub}. \cref{th:outerplanar-lb-3} proves that $c_k=\frac{k-1}{k}$ is indeed the right bound for $k=3$. Also, we suspect that these techniques could be extended to find large induced subgraphs of bounded degree in graphs of bounded treewidth.

In order to find large induced subgraphs of degree at most $k$ in a planar graph $G$, we employed a  greedy algorithm which repeatedly removes a vertex with maximum degree, until the graph has degree at most $k$. For $k=3$ and $k=4$, we modified the algorithm so that the vertex removal stops when $G$ has degree at most $k+1$. Then $G$ is augmented so that it becomes $(k+1)$-regular. Finally, $G$ achieves degree at most $k$ by the removal of a small dominating set, for which we have good upper bounds on the minimum size \cite{liu2004domination,xing2006upper}. 

It is worth remarking that the greedy algorithm without the dominating set strategy still produces a lower bound better than $\lceil \frac{n}{3}\rceil$, which is the size of a linear forest that one is guaranteed to find in $G$ \cite{DBLP:journals/jgt/Poh90}. However, for the case in which $d=4$ and $k=3$, it only gives us an $\frac{n}{2}$ lower bound, see \cref{fig:upper-bound-greedy}, which is worse than the $\frac{7n}{11}$ lower bound we get by using the dominating set strategy, see \cref{th:planar-lb}; consequently, not employing the dominating set strategy leads to worse bounds also for $d=5$ and $d\geq 6$. 
\begin{figure}
	\centering
	\includegraphics[scale=.7, page=1]{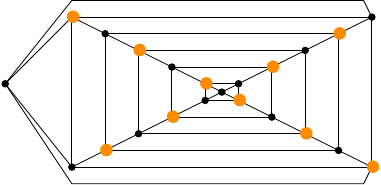}
	\caption{An $n$-vertex $4$-regular planar graph in which repeatedly removing degree-$4$ vertices might result in the removal of $\frac{n}{2}-O(1)$ vertices. The removed vertices are large disks.}
	\label{fig:upper-bound-greedy}
\end{figure}

We also ask whether one could obtain lower bounds better than ours for $k\geq 5$ by using known results on the size of dominating sets in $(k+1)$-regular graphs \cite{DBLP:books/daglib/0021015}.

Observe that the degree of a planar graph can be reduced to~$k$ by repeatedly applying the dominating set strategy, each time decreasing the degree of the current graph by $1$. This approach does not seem to lead to good bounds for our problem.

We remark that our greedy algorithm does not exploit the planarity of the graph, other than for bounding its number of edges. Hence, we find interesting to understand whether planarity can be exploited more deeply to get better bounds. In particular, lower bounds better than the ones presented in this paper could be obtained if one could prove that every planar graph with maximum degree $4$ or with maximum degree $5$ has a subset of vertices whose size is smaller than~$\frac{4n}{11}$ or~$\frac{5n}{14}$, respectively, such that every degree-$4$ vertex is in the set or adjacent to a vertex in the set.   

\vspace{3mm}
{\bf Acknowledgments.} Thanks to Ross J. Kang and Michael J. Pelsmajer for pointing out to us relevant literature for the problem we studied.


\bibliography{bibliography}

\end{document}